\documentclass[11pt]{article}

\usepackage{fancyhdr}
\usepackage{amsfonts}
\usepackage{amsmath}
\usepackage{amssymb}
\usepackage{amsthm}
\usepackage{tikz}

\usepackage{amscd}
\usepackage{amstext}

\usepackage{verbatim}

\usepackage[draft=false,colorlinks,bookmarksnumbered,linkcolor=black, citecolor=black]{hyperref}

\usepackage{oldgerm} 

\usepackage{fancyhdr}
\usepackage{amsfonts}
\usepackage{amsmath}
\usepackage{amssymb}
\usepackage{amsthm}

\usepackage{amscd}
\usepackage{amstext}

\usepackage{graphics}
\usepackage{graphicx}

\newlength{\fixboxwidth}
\newcommand*{\abs}[1]{\left| #1 \right|}                                
\newcommand*{\norm}[1]{\left\| #1 \right\|}                             
\newcommand*{\sep}{\; \vrule \;}                                        
\renewcommand{\d}{\,\mathrm{d}}											
\newcommand{\loc}{\mathrm{loc}}											

\newcommand{\re}{\mathbb{R}}\newcommand{\N}{\mathbb{N}}
\newcommand{\zz}{\mathbb{Z}}\newcommand{\C}{\mathbb{C}}

\newcommand{\com}{\mathbb{C}}

\newcommand{\R}{{\re}^d}

\newcommand{\cs}{{\mathcal S}}

\newcommand{\cl}{{\mathcal L}}

\newcommand{\cf}{{\mathcal F}}
\newcommand{\cfi}{{\cf}^{-1}}

\newcommand{\supp}{{\rm supp \, }}

\newcommand{\be}{\begin{equation}}
\newcommand{\ee}{\end{equation}}
\newcommand{\beq}{\begin{eqnarray}}
\newcommand{\beqq}{\begin{eqnarray*}}
\newcommand{\eeq}{\end{eqnarray}}
\newcommand{\eeqq}{\end{eqnarray*}}

\usepackage[draft=false,colorlinks,bookmarksnumbered,linkcolor=black, citecolor=black]{hyperref}

\usepackage{aliascnt}

\newtheorem{theorem}{Theorem}[section]

\newaliascnt{lem}{theorem}
\newtheorem{lemma}[lem]{Lemma}

\aliascntresetthe{lem}

\newaliascnt{ass}{theorem}

\aliascntresetthe{ass}

\newaliascnt{prop}{theorem}
\newtheorem{prop}[prop]{Proposition}

\aliascntresetthe{prop}

\newaliascnt{cor}{theorem}
\newtheorem{corollary}[cor]{Corollary}

\aliascntresetthe{cor}

\newaliascnt{defi}{theorem}
\newtheorem{defi}[defi]{Definition}

\aliascntresetthe{defi}

\theoremstyle{definition}
\newaliascnt{ex}{theorem}

\aliascntresetthe{ex}

\newaliascnt{rem}{theorem}
\newtheorem{remark}[rem]{Remark}

\aliascntresetthe{rem}

\begin{document}


\title{On the Boundedness of Dilation Operators in the Context of Triebel-Lizorkin-Morrey Spaces}
\author{Marc Hovemann\footnote{Friedrich-Schiller-Universität Jena, Fakultät für Mathematik und Informatik, Inselplatz 5, 07743 Jena, Germany. 
Email: \href{mailto:marc.hovemann@uni-jena.de}{marc.hovemann@uni-jena.de} } 
$^{,}$\footnote{Corresponding author.}
\quad\ and \quad Markus Weimar\footnote{Julius-Maximilians-Universit\"at W\"urzburg (JMU), 
Institute of Mathematics, 
Emil-Fischer-Stra{\ss}e 30, 97074 W\"urzburg, Germany. 
Email: \href{mailto:markus.weimar@uni-wuerzburg.de}{markus.weimar@uni-wuerzburg.de} }
}
\date{Dedicated to the 90th anniversary of Hans Triebel\\[10pt] \today}

\maketitle

\noindent\textbf{Abstract:} In this paper we study the behavior of dilation operators $  D_\lambda \colon f \mapsto f(\lambda\,\cdot) $ with $ \lambda > 1 $ in the context of Triebel-Lizorkin-Morrey spaces $\mathcal{E}^{s}_{u,p,q}(\mathbb{R}^d)$. For that purpose we prove upper and lower bounds for the operator (quasi-)norm $\norm{ D_\lambda \sep \mathcal{L}\big(\mathcal{E}^s_{u,p,q}(\R)\big) } $. 
We show that for $s>\sigma_p $ the operator (quasi-)norm $\norm{ D_\lambda \sep \mathcal{L}\big(\mathcal{E}^s_{u,p,q}(\R)\big) } $ up to constants behaves as $\lambda^{s - \frac{d}{u}} $. 
For the borderline case $ s = \sigma_{p} $ we observe a behavior of the form $\lambda^{\sigma_p- \frac{d}{u}}$, multiplied with logarithmic terms of $\lambda$ that also depend on the fine index $q$. 
For $s < \sigma_{p}$ and $p \geq 1$  we find the relation $\norm{ D_\lambda \sep \mathcal{L}\big(\mathcal{E}^s_{u,p,q}(\R)\big) } \sim \lambda^{ - \frac{d}{u}}$. 
The case $s < \sigma_{p}$ and $p < 1$ is investigated as well. Our proofs are mainly based on the Fourier analytic approach to Triebel-Lizorkin-Morrey spaces. As byproducts we show an advanced Fourier multiplier theorem for band-limited functions in the context of Morrey spaces and derive some new equivalent (quasi-)norms and characterizations of $\mathcal{E}^{s}_{u,p,q}(\mathbb{R}^d)$.

\vspace{0,2 cm}


\vspace{0,2 cm}


\section{Introduction and Main Results}

Nowadays Triebel-Lizorkin spaces $F^s_{p,q} (\R)$ are well-established tools to describe the regularity of functions and distributions beyond the classical scale of $L_p$-Sobolev spaces $H^s_p(\R)$ which are included as special cases with $q=2$. They have been introduced around 1970 by Lizorkin~\cite{Liz1, Liz2} and Triebel~\cite{Tr73}. 
Later on these function spaces have been investigated in detail in the famous books of Triebel~\cite{Tr83,Tr92,Tr06,Tr20} which contain a variety of other historical references. 
In recent years a growing number of authors works with further generalisations of the Triebel-Lizorkin scale defined upon \emph{Morrey spaces} $\mathcal{M}^u_p$ instead of classical Lebesgue spaces~$L_p$. 
So, \emph{Triebel-Lizorkin-Morrey spaces} 
$$
    \mathcal{E}^{s}_{u,p,q}(\R) \quad\text{with}\quad 0 < p \leq u < \infty,\; 0 < q \leq \infty \,\text{ and }\, s \in \mathbb{R}
$$ 
(see \autoref{def_tlm} below)
as well as Triebel-Lizorkin-type spaces $F^{s,\tau}_{p,q}(\R) $ with $ 0 < p < \infty $, $ 0 < q \leq \infty $, $ s \in \mathbb{R} $ and $ 0\leq \tau<\infty $ recently attracted a lot of attention. 
The spaces $\mathcal{E}^{s}_{u,p,q}(\R) $ have been introduced in 2005 by Tang and Xu~\cite{TangXu}, while  
$F^{s,\tau}_{p,q}(\R) $ showed up for the first time in 2008 in the papers of Yang and Yuan~\cite{yy1,yy2}. 
Later on, the latter also appeared in~\cite{Tr14}, where a different notation has been used. 
Although $ \mathcal{E}^{s}_{u,p,q}(\R) $ and  $ F^{s,\tau}_{p,q}(\R) $ are defined in a  quite different way, they have a lot of properties in common. 
Under certain conditions on the parameters they even coincide. 
In particular, both classes extend the ordinary Triebel-Lizorkin scale, since $F^s_{p,q} (\R) = \mathcal{E}^{s}_{p,p,q}(\R) = F^{s,0}_{p,q} (\R)$; cf.~\cite{ysy} for details.

In this paper we study the behavior of \emph{dilation operators} 
\begin{align}\label{eq:dilation}
    D_\lambda \colon f \mapsto f(\lambda\,\cdot) \quad\text{with}\quad \lambda>1
\end{align}
in the context of Triebel-Lizorkin-Morrey spaces $\mathcal{E}^{s}_{u,p,q}(\R)$. In connection with that we find upper and lower bounds for the operator (quasi-)norm $\norm{ D_\lambda \sep \mathcal{L}\big(\mathcal{E}^s_{u,p,q}(\R)\big) }$. It turns out that the shape of this norm heavily depends on the parameters $s, u$ and $ p $. In some borderline cases also the fine index $q$ is involved. For the range of Bessel potential spaces (i.e.\ fractional $L_p$-Sobolev spaces), it is well-known since many years that with constants independent of $\lambda>1$ there holds
\begin{align}\label{eq:norm_on_Hsp}
    \norm{D_\lambda \sep \mathcal{L}\big(H^s_p(\R)\big)} \sim \lambda^{\max\{s,0\}-\frac{d}{p}}, \qquad 1<p<\infty,\; s\in\mathbb{R}.
\end{align}
For the classical Besov spaces $B^s_{p,q}(\R)$ with $s > \sigma_p$ the behavior of $D_{\lambda}$ has been studied by Triebel~\cite[Ch.~3.4.1]{Tr83} and by Edmunds and Triebel~\cite[Ch.~2.3.1]{EdTri}. The borderline case $s = \sigma_{p}$ was treated by Vybíral \cite{Vyb08} (case $s=\sigma_p$ with $1\leq p \leq \infty$) as well as Schneider~\cite{Sch09} (case $s=\sigma_p$ with $0<p \leq 1$) for $\lambda = 2^{j}$ with $j \in \mathbb{N}$. 
In the context of Triebel-Lizorkin spaces $F^{s}_{p.q}(\mathbb{R}^d)$ again the boundedness of $\norm{ D_\lambda \sep \mathcal{L}\big(F^s_{p,q}(\R)\big) }$ has been observed by Triebel~\cite[Ch.~3.4.1]{Tr83} and by Edmunds and Triebel~\cite[Ch.~2.3.1]{EdTri} for $s > \sigma_p$. Later on, these results have been complemented by Schneider and Vybíral~\cite{SchVy}. They dealt with the limiting case $s=\sigma_p$, whereby $\lambda = 2^{j}$ with $j \in \mathbb{N}$. However, for $0<p\leq 1$ there remained some gaps between their upper and lower bounds.

It is our main goal to extend these operator (quasi-)norm estimates to the setting of general Triebel-Lizorkin-Morrey spaces $\mathcal{E}^{s}_{u,p,q}(\R)$ and arbitrary dilation factors $\lambda>1$. It turns out that the difficulty of this task heavily depends on the parameters $s, u, p$ and~$q$. Indeed, already in \cite[Sect.~5.1]{Ho1} it has been observed that it is rather easy to find upper bounds for $\norm{ D_\lambda \sep \mathcal{L}\big(\mathcal{E}^s_{u,p,q}(\R)\big) }$ if $s > \sigma_{p,q}$, since then the Triebel-Lizorkin-Morrey spaces can be described via differences of higher order. Moreover, in \cite[Lem.~5.3]{Ho1} some first attempts concerning $s=0$ have been made. On the other hand, in \cite[Prop.~2.15]{Saadi1} for the Banach space case $p, q \geq 1$ and $s > 0$ upper bounds have been proved by using the connection between the Triebel-Lizorkin-type spaces and their homogeneous counterparts. 
In our main result below we extend these findings to the full range of  parameters. Moreover, we provide both upper and lower bounds for $\norm{ D_\lambda \sep \mathcal{L}\big(\mathcal{E}^s_{u,p,q}(\R)\big) }$. 
Note that in any case the condition $0<p\leq u$ implies that $\sigma_p\geq \sigma_u$ (with equality iff $p=u$) and $\sigma_u \geq 0$; cf.~\eqref{eq:sigma_p}.
\begin{theorem}\label{thm:main}
    Let $0<p\leq u <\infty$, $0<q\leq\infty$ and $s\in\mathbb{R}$. 
    Then for $\frac{1}{2}<\lambda<2$ the restriction of the dilation operator $D_\lambda$ to the Triebel-Lizorkin-Morrey space $\mathcal{E}^s_{u,p,q}(\R)$ satisfies 
    $$
        \norm{D_\lambda \sep \mathcal{L}\big(\mathcal{E}^s_{u,p,q}(\R)\big)} \sim 1.
    $$
    For $\lambda\geq 2$ the following estimates hold:
    \begin{enumerate}
        \item If $s>\sigma_p$, then
        $$
            \norm{D_\lambda \sep \mathcal{L}\big(\mathcal{E}^s_{u,p,q}(\R)\big)} 
            \sim \lambda^{s - \frac{d}{u}}.
        $$

        \item If $s=\sigma_p$, then for $p>1$ it is
        $$
            \norm{D_\lambda \sep \mathcal{L}\big(\mathcal{E}^{\sigma_p}_{u,p,q}(\R)\big)} 
            \sim \lambda^{\sigma_p- \frac{d}{u}} \, (\log_2 \lambda)^{\max\{0,\frac{1}{q}-\frac{1}{2}\}}
        $$
        while for $p=1$ there holds
        $$
            \lambda^{\sigma_p- \frac{d}{u}} (\log_2 \lambda)^{\max\{\frac{1}{p}, \frac{1}{q}\}}
            \!\gtrsim\! \norm{D_\lambda \sep \mathcal{L}\big(\mathcal{E}^{\sigma_p}_{u,p,q}(\R)\big)}
            \!\gtrsim\! \lambda^{\sigma_p- \frac{d}{u}}
            \cdot \begin{cases}
                (\log_2 \lambda)^{\max\{\frac{1}{p}, \frac{1}{q}-\frac{1}{2}\}}, & p = u, \\
                (\log_2 \lambda)^{\max\{0,\frac{1}{q}-\frac{1}{2}\}},  & p<u,
            \end{cases}
        $$
        and for $p<1$ we have
        $$
            \lambda^{\sigma_p- \frac{d}{u}} (\log_2 \lambda)^{\max\{\frac{1}{p}, \frac{1}{q}\}}
            \gtrsim \norm{D_\lambda \sep \mathcal{L}\big(\mathcal{E}^{\sigma_p}_{u,p,q}(\R)\big)}
            \gtrsim \lambda^{\sigma_p- \frac{d}{u}}
            \cdot \begin{cases}
                (\log_2 \lambda)^{\frac{1}{p}},  & p = u, \\
                1, & p<u.
            \end{cases}
        $$

        \item If $s<\sigma_p$, then for $p\geq 1$ it holds
        $$
            \norm{D_\lambda \sep \mathcal{L}\big(\mathcal{E}^s_{u,p,q}(\R)\big)} 
            \sim \lambda^{\sigma_p - \frac{d}{u}}
        $$
        while for $p<1$ we have
        $$
            \lambda^{\sigma_p- \frac{d}{u}}
            \!\gtrsim \norm{D_\lambda \sep \mathcal{L}\big(\mathcal{E}^s_{u,p,q}(\R)\big)}
            \gtrsim \lambda^{\max\{s,\sigma_u\}- \frac{d}{u}}
            \cdot \begin{cases}
                (\log_2 \lambda)^{\max\{0, \frac{1}{q}-\frac{1}{2}\}}, & s \!=\! 0 \text{ and } 1 \!\leq\! u,\\
                1, &\text{else}.
            \end{cases}
        $$
    \end{enumerate}
    Therein, all implied constants are independent of $\lambda$. To incorporate the case $q = \infty$ we use the convention $\frac{1}{\infty}= 0$.
\end{theorem}
It turns out that for $s>\sigma_p $ the operator (quasi-)norm $\norm{ D_\lambda \sep \mathcal{L}\big(\mathcal{E}^s_{u,p,q}(\R)\big) }$ behaves as $\lambda^{s - \frac{d}{u}} $ (up to constants). For the borderline case $ s = \sigma_{p} $ we observe a behavior of the form $\lambda^{\sigma_p- \frac{d}{u}}$, multiplied with logarithmic terms in $\lambda$ that also depend on the fine index~$q$. For $s < \sigma_{p}$ and $p \geq 1$  we find the relation $\norm{ D_\lambda \sep \mathcal{L}\big(\mathcal{E}^s_{u,p,q}(\R)\big) } \sim \lambda^{ - \frac{d}{u}}$. The case $s < \sigma_{p}$ and $p < 1$ is more difficult, and perhaps the lower bounds can be improved further. We note in passing that an estimate of the form $\norm{D_{2^j} \sep \mathcal{L}\big(\mathcal{E}^s_{u,p,q}(\R)\big)} \lesssim 2^{j(s-\frac{d}{u})}$ for some $0<p\leq u < \infty$, $0<q\leq\infty$, $s\in\mathbb{R}$ and all $j\in\N$ can hold only if $s\geq \sigma_u$ as otherwise it would contradict \autoref{thm:main}(iii).

Specializing \autoref{thm:main} to the case $u:=p$ we obtain the following results for the scale of classical Triebel-Lizorkin spaces $F^s_{p,q}(\R)=\mathcal{E}^s_{p,p,q}(\R)$ which particularly also cover~\eqref{eq:norm_on_Hsp} by choosing $p>1$ and $q:=2$. 
If $\lambda=2^j$ with $j\in\N$ and $s\geq \sigma_p$ our estimates exactly coincide with the findings in \cite{EdTri} and \cite{SchVy}, in this regard see also \autoref{rem:correctionFspq} below. Otherwise, they seem to be new.

\begin{corollary}\label{cor:main}
    Let $0<p<\infty$, $0<q\leq\infty$ and $s\in\mathbb{R}$. Then for $\frac{1}{2}<\lambda<2$ we have
    $$
        \norm{D_\lambda \sep \mathcal{L}\big(F^s_{p,q}(\R)\big)} \sim 1,
    $$
    while for $\lambda\geq 2$ there holds
    $$
            \norm{D_\lambda \sep \mathcal{L}\big(F^s_{p,q}(\R)\big)}
            \sim \lambda^{\max\{s,\sigma_p\}- \frac{d}{p}} \qquad\text{if}\qquad s\neq \sigma_p 
    $$
    as well as
    $$
            \norm{D_\lambda \sep \mathcal{L}\big(F^{\sigma_p}_{p,q}(\R)\big)}
            \sim \lambda^{\sigma_p- \frac{d}{p}} \, (\log_2 \lambda)^{\max\{0, \frac{1}{q}-\frac{1}{2}\}} \qquad\text{if}\qquad p> 1 
    $$
    and
    $$
            \lambda^{\sigma_p- \frac{d}{p}} (\log_2 \lambda)^{\max\{\frac{1}{p}, \frac{1}{q}\}}
            \gtrsim \norm{D_\lambda \sep \mathcal{L}\big(F^{\sigma_p}_{p,q}(\R)\big)}
            \gtrsim \lambda^{\sigma_p- \frac{d}{p}} \cdot \begin{cases}
                (\log_2 \lambda)^{\max\{\frac{1}{p} , \frac{1}{q}-\frac{1}{2}\}}, & p=1,\\
                (\log_2 \lambda)^{\frac{1}{p}}, & p < 1.
            \end{cases} 
    $$
    Again, all implied constants are independent of $\lambda$. To incorporate the case $q = \infty$ we use the convention $\frac{1}{\infty}= 0$.
\end{corollary}

Dilation operators $D_{\lambda}$ of the form \eqref{eq:dilation} have a lot of significant applications within the theory of function spaces, and also in related fields of research. For example, they have been widely used in the regularity analysis of PDEs as an important tool in order to construct corresponding (appropriately weighted) function spaces; see, e.g.\ \cite{CioSchWei, Kry, Lot, MazRos}.
Furthermore, the operators $D_\lambda$ play a crucial role in the theories of so-called refined localization spaces \cite[Sect.~2.2]{Tr08} and tempered homogeneous spaces \cite{Tr15} due to Triebel. In addition, sharp bounds of the operator norm $\norm{ D_\lambda \sep \mathcal{L}\big(\mathcal{E}^s_{u,p,q}(\R)\big) } $ can be used to disprove the equivalence of certain quasi-norms, at least if specific conditions on the parameters hold. For instance, in \cite[Sect.~5.1]{Ho1} and \cite[Ch.~6.1]{H21} a result comparable with  \autoref{thm:main}(ii) has been used to show that for $s = 0$ the Triebel-Lizorkin-Morrey spaces can not be equivalently described via ball means of higher order differences. 

This paper is organized as follows. 
In \autoref{subsec_dil_basic} we collect some basic properties of the dilation operators $D_{\lambda}$. Throughout Sections \ref{subsec_Morrey} and \ref{Sec_Defi} we recall the definitions and fundamental facts about Morrey and Triebel-Lizorkin-Morrey spaces, respectively. In connection with that in \autoref{prop:FM_M} we also prove a new Fourier multiplier theorem for band-limited functions in the context of Morrey spaces. 
\autoref{sec_main_sec_main} is devoted to the proof of our main \autoref{thm:main}. 
For that purpose, in \autoref{sec_dilation_near1} we show that it is sufficient to prove our statements only for dyadic dilation operators of the form $D_{2^{j}}$ with $j \in \mathbb{N}$. Afterwards, in \autoref{sec_upper_bounds_1} all upper bounds stated in \autoref{thm:main} will be proven. The corresponding lower bounds are derived in \autoref{sec_lower_bounds_1} and in \autoref{sec_put_together_1} all results are combined in order to receive \autoref{thm:main} and \autoref{cor:main}, respectively. 
Finally, based on these findings, we construct some new equvialent (quasi-)norms for $\mathcal{E}^s_{u,p,q}(\R)$ in the last \autoref{sect:characterizations}; see Theorems \ref{thm:characterization} and \ref{thm:characterization_by_M}.
Some of them can even be used as characterizations. 

\medskip

\noindent\textbf{Notation.} As usual, $\N$ denotes the natural numbers, $\N_0:=\N\cup\{0\}$, $\zz$ describes the integers and~$\re$ the real numbers. Further, $\R$ with $d\in\N$ denotes the $d$-dimensional Euclidean space and we put
$$
    B(x,t) := \left\{y\in \R \,:\, \abs{x-y}< t \right\}\, , \qquad x \in \R,\,\; t>0.
$$
All our functions and distributions are assumed to take values in the complex numbers $\com$. 
In particular, we let $M(\R)$ be the space of equivalence classes of measurable functions on~$\R$ (w.r.t.\ equality a.e.), $\mathcal{S}(\R)\subseteq M(\R)$ denotes the collection of all Schwartz functions and
$\mathcal{D}(\R):=C_0^\infty(\R) \subseteq \mathcal{S}(\R)$ is the set of infinitely often differentiable functions with compact support.
For domains (open connected sets) $\Omega\subseteq\R$ and $0<p<\infty$, we let $L_p^\loc(\Omega)$ be the set of measurable and locally $p$-integrable functions on $\Omega$ while its subset $L_p(\Omega)\subseteq L_p^\loc(\Omega)$ denotes the ordinary Lebesgue space.
The sets $\mathcal{D}(\R)$ and $\mathcal{S}(\R)$ are endowed with the usual topology and we write $\mathcal{D}'(\R)$ resp.\ $\mathcal{S}'(\R)$ for their topological duals, i.e.\ the spaces of corresponding distributions (continuous linear functionals w.r.t.\ the {weak-$\ast$} topology).
The symbol $\cf$ refers to the Fourier transform and $\cfi$ to its inverse, both defined on $\cs'(\R)$. 
The (quasi-)norm in a (quasi-)Banach space $X$ is denoted by $\norm{\cdot \sep X}$ and by $\cl (X)$ we denote the space of all linear bounded operators $T\colon X\to X$ on it.
Moreover, given two spaces $X$ and $Y$, we write $X \hookrightarrow Y$ if the natural embedding of $X$ into $Y$ is continuous. 
In addition, we shall use the well-established quantity
\begin{align}\label{eq:sigma_p}
    \sigma_p:= d\,  \max\!\left\{0, \frac 1p - 1\right\}, \qquad 0<p<\infty.
\end{align}
With $A \lesssim B$ we mean $A \leq c B$ with some constant $c > 0$ independent of $A$ and $B$. Finally, the notation $ A \sim B $ stands for $A \lesssim B$ and $B \lesssim A$.

\section{Preliminaries}\label{sect:prelim}
Let us start with a brief overview of our objects of interest.

\subsection{Basic Properties of Dilation Operators \texorpdfstring{$D_\lambda$}{Dlambda} with \texorpdfstring{$\lambda>0$}{lambda>0}}\label{subsec_dil_basic}
For factors $\lambda>0$ we let
$$
    D_\lambda \colon M(\R) \to M(\R), \qquad g \mapsto D_\lambda g := g(\lambda \,\cdot),
$$
denote the corresponding dilation operator for measurable functions. 
Then it is obvious that for $0<p<\infty$ and $X\in\{\mathcal{D}, \mathcal{S}, L_p^\loc\}$ the restriction of $D_\lambda$ is a linear map from~$X(\R)$ to itself.
Therefore, we can extend (the restriction of) $D_\lambda$ to spaces of distributions $Y'(\R)$, where $Y\in \{\mathcal{D}, \mathcal{S}\}$, by setting
\begin{align*}
    (D_\lambda f)(\varphi) := \lambda^{-d} f(D_{\lambda^{-1}} \varphi), \qquad f\in Y'(\R),\;  \varphi\in Y(\R),
\end{align*}
such that also $D_{\lambda} \colon Y'(\R)\to Y'(\R)$.
Note that if $f\in Y'(\R)$ is regular, i.e.\ induced by some $\widetilde{f}\in L_1^\loc(\R)$, then this definition yields that also $D_\lambda f$ is regular and induced by the function~$D_\lambda \widetilde{f}\in L_1^\loc(\R)$. For the sake of simplicity all restrictions or extensions of $D_\lambda$ will be denoted by $D_\lambda$ again. 

Let us collect some useful properties.
First of all note that in all settings $D_\lambda$ is bijective on $X(\R)$ resp.\ $Y'(\R)$ with $(D_\lambda)^{-1}=D_{\lambda^{-1}}$. 
Since obviously $D_1=\mathrm{id}$, the set $\{D_\lambda : \lambda > 0\}$ together with concatenation actually forms a group.
In particular, we have 
$$
    D_{\lambda_1}\circ D_{\lambda_2}=D_{\lambda_1\cdot\lambda_2}, \qquad \lambda_1,\lambda_2>0,
$$
such that, if $\psi_\ell := D_{2^{-\ell}}\psi$ for some $\psi\in \mathcal{D}(\R)$ and all $\ell\in\mathbb{Z}$, there holds
$$
    D_{2^j}\psi_k = D_{2^{-(k-j)}}\psi = \psi_{k-j}, \qquad j,k\in\mathbb{Z}.
$$ 
Second, for $\varphi \in \mathcal{S}(\R)$ and $\lambda>0$ it is easy to show that
\begin{align}
    \label{eq:FTj_phi}
    \mathcal{F}^{\pm 1}(D_\lambda \varphi) = \lambda^{-d} \, D_{\lambda^{-1}} [ \mathcal{F}^{\pm 1}\varphi ] \in \mathcal{S}(\R)
\end{align}
such that consequently for $f\in \mathcal{S}'(\R)$ also
\begin{align*}
    \mathcal{F}^{\pm 1} (D_\lambda f) = \lambda^{-d} \, D_{\lambda^{-1}}[\mathcal{F}^{\pm 1}f] \in \mathcal{S}'(\R).
\end{align*}
Finally, for $\eta\in\mathcal{D}(\R)$ and $f\in \mathcal{S}'(\R)$ straightforward calculations show that for all $\lambda>0$
\begin{align}
    \label{eq:building_block_shift}
    \mathcal{F}^{-1} \big( \eta \, \mathcal{F} [D_\lambda f]\big) 
    = D_\lambda \big[\mathcal{F}^{-1} \big( [D_\lambda\eta] \, \mathcal{F} f\big) \big]
\end{align}
as elements of $\mathcal{S}'(\R)$. 
However, using the famous Paley-Wiener-Schwartz Theorem (see e.g.\ Triebel~\cite[Thm.~1.2.1/2]{Tr83}), we find that these distributions can be identified with entire analytic functions restricted to $\R$. Consequently \eqref{eq:building_block_shift} also holds pointwise.

\subsection{Morrey Spaces \texorpdfstring{$\mathcal{M}^{u}_{p}(\R)$}{Mpu(Rd)}}\label{subsec_Morrey}
The Triebel-Lizorkin-Morrey spaces $\mathcal{E}^{s}_{u,p,q}(\R)$ we are interested in are spaces of tempered distributions built upon Morrey spaces~$\mathcal{M}^{u}_{p}(\R)$. Therefore, let us recall the definition and basic properties of the latter. 
\begin{defi}[Morrey space $\mathcal{M}^{u}_{p}(\R)$]
\label{def_mor}
    For $0 < p \leq u < \infty$ the Morrey space $\mathcal{M}^{u}_{p}(\R)$ is the collection of all functions $f \in L_{p}^{\loc}(\R)$ such that
    \begin{align*}
        \norm{f \sep \mathcal{M}^{u}_{p}(\R)}
        := \sup_{y \in \R, R > 0} \abs{B(y,R)}^{\frac{1}{u}-\frac{1}{p}} \left( \int_{B(y,R)} \abs{f(x)}^{p} \d x  \right)^{\frac{1}{p}} < \infty.
    \end{align*} 
\end{defi}
The Morrey spaces $\mathcal{M}^{u}_{p}(\R)$ are known to be translation-invariant quasi-Banach spaces (and even Banach spaces if $p \geq 1$). 
They have many connections to ordinary Lebesgue spaces $ L_{p}(\R)$. 
Indeed, for $ 0 < p_{2} \leq p_{1} \leq u < \infty $ we have
\begin{align}
    \label{eq:morrey_embedding}
    L_{u}(\R) 
    = \mathcal{M}^{u}_{u}(\R) 
    \hookrightarrow \mathcal{M}^{u}_{p_{1}}(\R)
    \hookrightarrow \mathcal{M}^{u}_{p_{2}}(\R).
\end{align}
Moreover, we shall frequently use the following simple properties:
\begin{lemma}\label{lem:M}
    Let $0 < p \leq u < \infty$ as well as $0<r<\infty$ and $\lambda>0$.
    \begin{enumerate}
        \item If $f\colon\R\to\C$ is measurable and $g\in \mathcal{M}^{u}_{p}(\R)$ such that $\abs{f(x)} \leq \abs{g(x)}$ for a.e.\ $x\in\R$, then $f\in \mathcal{M}^{u}_{p}(\R)$ and
        $
            \norm{f \sep \mathcal{M}^{u}_{p}(\R)} 
            \leq \norm{g \sep \mathcal{M}^{u}_{p}(\R)}.
        $

        \item There holds $f\in \mathcal{M}^{u}_{p}(\R)$ if and only if $\abs{f}^r \in \mathcal{M}^{\frac{u}{r}}_{\frac{p}{r}}(\R)$. In this case we have
        $$
            \norm{\abs{f}^r \sep \mathcal{M}^{\frac{u}{r}}_{\frac{p}{r}}(\R)}
            = \norm{f \sep \mathcal{M}^{u}_{p}(\R)}^r.
        $$

        \item We have $f\in \mathcal{M}^{u}_{p}(\R)$ if and only if $D_\lambda f \in \mathcal{M}^{u}_{p}(\R)$. In this case there holds
        $$
            \norm{D_\lambda f \sep \mathcal{M}^{u}_{p}(\R)} 
            \sim \lambda^{-\frac{d}{u}} \norm{ f \sep \mathcal{M}^{u}_{p}(\R)}
        $$
        with constants that do not depend on $\lambda$ and $f$.

        \item For $f,g\in \mathcal{M}^{u}_{p}(\R)$ and $0< \tau \leq \min\{1,p\}$  we have
        $$
            \norm{f+g \sep \mathcal{M}^{u}_{p}(\R)}^\tau 
            \leq \norm{f \sep \mathcal{M}^{u}_{p}(\R)}^\tau + \norm{g \sep \mathcal{M}^{u}_{p}(\R)}^\tau.
        $$
    \end{enumerate}
\end{lemma}
\begin{proof}
    The first assertion is obvious. For (ii) we refer to \cite[Lem.~3(vi)]{HoWei1} and for (iii) to \cite[Formula~(3.20)]{Sic}. For part~(iv) we note that $\abs{f(x)+g(x)}^\tau \leq \abs{f(x)}^\tau + \abs{g(x)}^\tau$ for a.e.\ $x\in\R$, as $\tau\leq 1$. Hence, from (ii),  (i) and $\frac{p}{\tau}\geq 1$ it follows 
    \begin{align*}
        \norm{f+g \sep \mathcal{M}^{u}_{p}(\R)}^\tau
        = \norm{ \abs{f+g}^\tau \sep \mathcal{M}^{\frac{u}{\tau}}_{\frac{p}{\tau}}(\R)}
        &\leq \norm{ \abs{f}^\tau \sep \mathcal{M}^{\frac{u}{\tau}}_{\frac{p}{\tau}}(\R)} + \norm{\abs{g}^\tau \sep \mathcal{M}^{\frac{u}{\tau}}_{\frac{p}{\tau}}(\R)} \\
        &= \norm{ f \sep \mathcal{M}^{u}_{p}(\R)}^\tau + \norm{g \sep \mathcal{M}^{u}_{p}(\R)}^\tau.
    \end{align*}
    So the proof is complete.
\end{proof}

Besides these rather elementary properties, we will employ Fourier multipliers for Morrey spaces.
The following vector-valued assertion was shown by Tang and Xu~\cite{TangXu} using maximal operators. Therein, $H^\nu(\R)$ denotes the $L_2$-Bessel potential space of order $\nu\in\mathbb{R}$.
\begin{prop}[{\cite[Thm.~2.7]{TangXu}}]\label{prop:vv-FM}
    Let $0 < p \leq u < \infty$ as well as $0 < q \leq \infty$ and let $(f_k)_{k\in\N_0}\subseteq \mathcal{S}'(\R)$ such that $f_k\in L_p(\R)$ and $\supp(\mathcal{F}f_k)\subseteq \Omega_k$, where $\Omega_k\subseteq \R$ is compact with radius $d_k>0$, $k\in\N_0$. 
    Further, let $(M_k)_{k\in\N_0}\subseteq H^\nu(\R)$, where $\nu > \frac{d}{2}+\frac{d}{\min\{p,q\}}$. Then
    \begin{align*}
        &\norm{\bigg( \sum_{k=0}^\infty \abs{\big(\mathcal{F}^{-1}[M_k\, \mathcal{F}
        f_k]\big)(\cdot)}^q \bigg)^{\frac{1}{q}} \sep \mathcal{M}^u_p(\R)} \\
        &\qquad \lesssim \sup_{k\in\N_0} \norm{D_{d_k}M_k \sep H^\nu(\R)} \, \norm{\bigg( \sum_{k=0}^\infty \abs{f_k(\cdot)}^q \bigg)^{\frac{1}{q}} \sep \mathcal{M}^u_p(\R)}
    \end{align*}
    (usual modification if $q=\infty$) with an implied constant that does not depend on $ (f_k)_{k\in\N_0} $, $ (M_k)_{k\in\N_0} $ or $ (d_k)_{k\in\N_0} $.
\end{prop}

Obviously, \autoref{prop:vv-FM} particularly covers an ordinary Fourier multiplier statement about single functions $f$ and $M$ as special case.
However, in view of the applications we have in mind, we shall need a stronger assertion that avoids Bessel potential norms of high smoothness.
The subsequent result for band-limited functions extends \cite[Prop.~1.5.1]{Tr83} from $L_p(\R)$ to Morrey spaces $\mathcal{M}^{u}_{p}(\mathbb{R}^d)$ and gives better estimates later on.
\begin{theorem}\label{prop:FM_M}
    Let $0 < p \leq u < \infty$ and $m,\ell>0$. Then for all $M \in \mathcal{S}(\mathbb{R}^d)$ with $\supp M \subseteq B(0,m)$ and every $f \in \mathcal{M}^{u}_{p}(\mathbb{R}^d)$ with $\supp(\mathcal{F}f) \subseteq B(0,\ell)$ there holds
    $$
        \norm{\mathcal{F}^{-1} [M\, \mathcal{F}f] \sep \mathcal{M}^{u}_{p}(\mathbb{R}^d)} 
        \lesssim (m+\ell)^{\sigma_p} \norm{\mathcal{F}^{-1}M \sep L_{\min\{1,p\}}(\R)} \norm{f \sep \mathcal{M}^{u}_{p}(\mathbb{R}^d)}
    $$
    with an implied constant that does not depend on $f$, $M$, $m$ or $\ell$.
\end{theorem}
\begin{proof}
    At first, we note that \cite[Formula~(7)]{SawTan} implies $f\in L_\infty(\R)$. Furthermore, we have $\mathcal{F}^{-1}M \in \mathcal{S}(\R) \subseteq L_1(\R)\cap L_p(\R)$ such that for a.e.\ $x \in \mathbb{R}^d$ there holds
    \begin{align}\label{eq:proof_Faltung}
        \mathcal{F}^{-1}[M \, \mathcal{F}f](x) 
        = c \int_{\mathbb{R}^d} F_x(y) \d y
        \quad\text{with}\quad  F_x := (\mathcal{F}^{-1}M)(\cdot) \, f(x-\cdot) \in L_1(\R)\cap L_p(\R)
    \end{align}
    and some $c>0$; cf.\ \cite[Rem.~1.5.1/1]{Tr83}. Now we distinguish two cases.

    To start with, consider the case $p\geq 1$ such that $\min\{1,p\}=1$ and $\sigma_p=0$. Then \eqref{eq:proof_Faltung} and Minkowski's inequality show that
    \begin{align*}
        \norm{\mathcal{F}^{-1} [M\, \mathcal{F}f] \sep \mathcal{M}^{u}_{p}(\mathbb{R}^d)} 
        &= \sup_{y \in \R, r > 0} \abs{B(y,r)}^{\frac{1}{u}-\frac{1}{p}} \left( \int_{B(y,r)} \abs{c \int_{\mathbb{R}^d} F_x(y) \d y}^{p} \d x  \right)^{\frac{1}{p}} \\
        &\lesssim \sup_{y \in \R, r > 0} \abs{B(y,r)}^{\frac{1}{u}-\frac{1}{p}} \left( \int_{B(y,r)} \bigg( \int_{\mathbb{R}^d} \abs{ F_x(y)} \d y \bigg)^{p} \d x  \right)^{\frac{1}{p}} \\
        &\leq \sup_{y \in \R, r > 0} \abs{B(y,r)}^{\frac{1}{u}-\frac{1}{p}} \int_{\R} \left( \int_{B(y,r)} \abs{ F_x(y) }^p \d x \right)^{\frac{1}{p}} \d y \\
        &\leq \int_{\R} \abs{(\mathcal{F}^{-1}M)(y)} \norm{ f(\cdot-y) \sep \mathcal{M}^{u}_{p}(\mathbb{R}^d)} \d y
    \end{align*}
    and the claim follows from the translation invariance of $\mathcal{M}^{u}_{p}(\mathbb{R}^d)$.

    To complete the proof, now let $0<p < 1$, i.e.\ $\min\{1,p\}=p$ and $\sigma_p=d(\frac{1}{p}-1)$.
    Then, following the lines of the proof of \cite[Prop.~1.5.1]{Tr83}, for a.e.\ $x\in\R$ we have 
    $$
        \supp(\mathcal{F} F_x) 
        \subseteq \supp M + \supp\big(\mathcal{F}f(x-\cdot)\big)
        = \supp M - \supp(\mathcal{F}f) 
        \subseteq B(0,m+\ell)
    $$
    such that Nikol'skij's inequality (see \cite[Rem.~1.4.1/4 and (5) in 1.3.2]{Tr83}) yields
    $$
        \norm{F_x \sep L_1(\R)} 
        \lesssim (m+\ell)^{d(\frac{1}{p}-1)} \norm{F_x \sep L_p(\R)}.
    $$
    Combined with \eqref{eq:proof_Faltung} and Fubini's Theorem we arrive at
    \begin{align*}
        \norm{\mathcal{F}^{-1} [M\, \mathcal{F}f] \sep \mathcal{M}^{u}_{p}(\mathbb{R}^d)} 
        &= \sup_{y \in \R, r > 0} \abs{B(y,r)}^{\frac{1}{u}-\frac{1}{p}} \left( \int_{B(y,r)} \abs{c \int_{\mathbb{R}^d} F_x(y) \d y}^{p} \d x  \right)^{\frac{1}{p}} \\
        &\lesssim (m+\ell)^{\sigma_p} \sup_{y \in \R, r > 0} \abs{B(y,r)}^{\frac{1}{u}-\frac{1}{p}} \left( \int_{B(y,r)} \int_{\mathbb{R}^d} \abs{ F_x(y) }^p \d y  \d x  \right)^{\frac{1}{p}} \\
        &\leq (m+\ell)^{\sigma_p} \bigg( \int_{\R} \abs{(\mathcal{F}^{-1}M)(y)}^p \norm{ f(\cdot-y) \sep \mathcal{M}^{u}_{p}(\mathbb{R}^d)}^p \d y \bigg)^{\frac{1}{p}}
    \end{align*}
    and the assertion follows as before from the translation invariance of $\mathcal{M}^{u}_{p}(\mathbb{R}^d)$.
\end{proof}

The last key property of Morrey spaces $\mathcal{M}^u_p(\R)$ we will need is a generalized Littlewood-Paley type assertion due to Izumi, Sawano and Tanaka~\cite{IzuSawTan2015}.
To formulate it, let $1\leq a < b \leq 2$ be fixed and choose a non-negative smooth function $\varphi_0 \in \mathcal{D}({\R})$ such that $\varphi_0(x) = 1$ if $\abs{x} \leq a$ and $ \varphi_0 (x) = 0$ if $\abs{x} \geq b $. Further, let 
\begin{align}\label{def_mrounity_0}
    \psi(x) := \varphi_0(x) - (D_2\varphi_0)(x) = \varphi_0(x) - \varphi_0(2x),\qquad  x \in \R,
\end{align}
as well as its dyadic dilates
\begin{align}\label{def_mrounity_1}
    \psi_\ell(x) := (D_{2^{-\ell}}\psi)(x)=\psi(2^{-\ell}x),\qquad  \ell\in\mathbb{Z},\;x \in \R,
\end{align}
such that $\supp \psi_\ell\subseteq \left\{x\in\R \,:\, 2^{\ell-1}a \leq \abs{x} \leq 2^\ell b\right\}$ with
\begin{align*}
    \psi_\ell \equiv 1
    \;\text{ on }\;
    \left\{x\in\R \,:\, 2^{\ell-1}b \leq \abs{x} \leq 2^\ell a\right\}.
\end{align*}
Hence, $\psi,\psi_\ell\in \mathcal{D}(\R)$ for every $\ell\in\mathbb{Z}$, such that 
for all $f\in\mathcal{S}'(\R)$ the distributions $\mathcal{F}^{-1}(\psi_\ell\, \mathcal{F} f)\in\mathcal{S}'(\R)$ are actually smooth functions on $\R$.

\begin{prop}[{\cite[Thm.~1.1]{IzuSawTan2015}}]\label{prop:LittlewoodPaley_M}
    Let $1<p\leq u < \infty$. Then for $f \in \mathcal{M}^u_p(\R)$ we have
    $$
        \norm{f \sep \mathcal{M}^u_p(\R)} 
        \sim \norm{\bigg( \sum_{\ell=-\infty}^\infty \abs{ \big(\mathcal{F}^{-1} [\psi_\ell \, \mathcal{F}f]\big)(\cdot)}^2 \bigg)^{1/2} \sep \mathcal{M}^u_p(\R)}
    $$
    with constants independent of $f$.
\end{prop}
\begin{proof}
    Note that the function $\overline{\varphi}_0:=D_{2^{-1}}\varphi_0 \in\mathcal{D}(\R)$  satisfies $\chi_{B(0,2)} \leq \overline{\varphi}_0 \leq \chi_{B(0,4)}$ on~$\R$. Then we observe that $\overline{\psi} := \overline{\varphi}_0-\overline{\varphi}_0(2\,\cdot) = D_{2^{-1}}(\varphi_0- D_2\varphi_0) =  D_{2^{-1}}\psi$ and thus $\overline{\psi}_\ell := D_{2^{-\ell}}\overline{\psi} = D_{2^{-(\ell+1)}}\psi = \psi_{\ell+1}$ for $\ell\in\mathbb{Z}$.
    Therefore, our claim follows from \cite[Thm.~1.1]{IzuSawTan2015} by a simple index shift.
\end{proof}

\subsection{Triebel-Lizorkin-Morrey Spaces \texorpdfstring{$\mathcal{E}^{s}_{u,p,q}(\R)$}{Mpu(Rd)}}\label{Sec_Defi}

In order to define the Triebel-Lizorkin-Morrey spaces $\mathcal{E}^{s}_{u,p,q}(\R)$ we require a so-called smooth dyadic decomposition of unity which is given by the family of smooth compactly supported functions
$$
    \varphi_0
    \qquad\text{and}\qquad \varphi_k:=\psi_k, \quad k\in\N,
$$
as defined above, see \eqref{def_mrounity_1} and \eqref{def_mrounity_0}.
Then 
\begin{align*}
    \sum_{k=0}^\infty \varphi_k(x) = 1, \qquad x\in \R, 
\end{align*}
and $\supp \varphi_k \subseteq \left\{x\in \R \,:\, 2^{k-1} a \leq \abs{x} \leq 2^{k}b \right\}$, $k \in \N$, which justifies the name smooth dyadic decomposition of unity for the system $(\varphi_k)_{k\in \N_0 }$. 

\begin{defi}[Triebel-Lizorkin-Morrey space $\mathcal{E}^{s}_{u,p,q}(\R)$]
\label{def_tlm}
    Let $ s \in \mathbb{R}$, $ 0 < p \leq u < \infty$ and $0 < q \leq \infty$. Further, let $ (\varphi_{k})_{k\in \N_0 }$ be a smooth dyadic decomposition of unity. Then the Triebel-Lizorkin-Morrey space $  \mathcal{E}^{s}_{u,p,q}(\mathbb{R}^{d})$ collects all $ f \in \mathcal{S}'(\mathbb{R}^{d})$ for which
    \begin{align*} 
        \norm{ f \sep \mathcal{E}^{s}_{u,p,q}(\mathbb{R}^{d}) } 
        := \norm{ \bigg( \sum_{k = 0}^{\infty} 2^{ksq} \abs{ \big(\mathcal{F}^{-1}[\varphi_{k}\, \mathcal{F}f] \big)(\cdot)}^{q} \bigg)^{\frac{1}{q}} \sep \mathcal{M}^{u}_{p}(\R) } < \infty.
    \end{align*}
    If $ q = \infty$, the usual modifications are made.
\end{defi}

Let us collect some well-known basic properties of Triebel-Lizorkin-Morrey spaces. 
\begin{lemma}\label{l_bp1}
    For $ 0 < p \leq u < \infty $, $ 0 < q \leq \infty $ and $ s \in \mathbb{R} $ the following holds true.
    \begin{enumerate}
        \item $\mathcal{E}^{s}_{u,p,q}(\mathbb{R}^{d})$ is independent of the chosen smooth dyadic decomposition of unity in the sense of equivalent quasi-norms. 
        
        \item The spaces $  \mathcal{E}^{s}_{u,p,q}(\mathbb{R}^{d}) $ are quasi-Banach spaces. For $ p,q \geq 1 $ they are Banach spaces.
        
        \item $\mathcal{S}(\mathbb{R}^{d}) \hookrightarrow    \mathcal{E}^{s}_{u,p,q}(\mathbb{R}^{d}) \hookrightarrow   \mathcal{S}'(\mathbb{R}^{d})$.

        \item $\mathcal{E}^{s}_{p,p,q}(\mathbb{R}^{d}) = F^{s}_{p,q}(\R)$.
    \end{enumerate}
\end{lemma} 
\begin{proof}
    Assertion~(i) was proved in \cite[Thm.~2.8]{TangXu} using \autoref{prop:vv-FM}. 
    The proof of~(ii) is standard; we refer to \cite[Lem.~2.1]{ysy}. 
    Also~(iii) with a slightly different formulation is proven in \cite[Prop.~2.3]{ysy}. 
    Finally, (iv) is obvious.
\end{proof}

A convenient way to show membership of distributions in $\mathcal{E}^s_{u,p,q}(\mathbb{R}^d)$ is given by the so-called \emph{dyadic annuli criterion} which has been proved by Yuan, Sickel and Yang~\cite{ysy}:
\begin{prop}[Cf.\ {\cite[Prop.~6.3]{ysy}}]\label{prop:dyadic_crit}
    Let $0 < p \leq u < \infty$, $0 < q \leq \infty$ and $s\in \mathbb{R}$. 
    Further, let $(\widetilde{u}_k)_{k\in\N_0}\subseteq \mathcal{S}'(\R)$ satisfy $\supp(\mathcal{F}\widetilde{u}_0)\subseteq B(0,4)$,
    $$
        \supp(\mathcal{F}\widetilde{u}_k) \subseteq B(0,2^{k+2})\setminus B(0,2^{k-3}), \qquad k\in\N, 
    $$
    and
    $$
        A:= \norm{\bigg( \sum_{k=0}^\infty 2^{ksq} \abs{\widetilde{u}_k(\cdot)}^q \bigg)^{\frac{1}{q}} \sep \mathcal{M}^u_p(\R)} < \infty
    $$
    (with the usual modification if $q=\infty$).
    Then $\sum_{k=0}^\infty \widetilde{u}_k$ converges in $\mathcal{S}'(\R)$ to some $U\in \mathcal{S}'(\R)$ and there holds
    $$
        U \in \mathcal{E}^s_{u,p,q}(\R) 
        \qquad\text{with}\qquad
        \norm{U \sep \mathcal{E}^s_{u,p,q}(\R)} \lesssim A,
    $$
    where the implied constant does not depend on $ (\widetilde{u}_k)_{k\in\N_0}$ or $A$. If $q<\infty$, the convergence takes place in $\mathcal{E}^s_{u,p,q}(\R)$ and otherwise in $\mathcal{E}^{s-\varepsilon}_{u,p,1}(\R)$ for all $\varepsilon>0$.
\end{prop}

In addition, we shall use a local mean characterization due to Sawano and Tanaka~\cite{SawTan}. Therein, $\Delta^{2N}$ denotes the $2N$-fold application of the Laplace operator in dimension $d$.
\begin{prop}[{\cite[Thm.~4.1]{SawTan}}]\label{prop:local_mean}
    Let $0<p\leq u < \infty$, $0<q\leq\infty$ and $s\in\mathbb{R}$. Further let $\eta \in\mathcal{S}(\R)$ such that $\chi_{B(0,1)} \leq \eta \leq \chi_{B(0,2)}$ and define $K:=\Delta^{2N}\eta$ for some sufficiently large $N=N_{u,p,q,s}\in\N$. Then with
    $$
        K_0:=\eta
        \qquad\text{and}\qquad 
        K_k := 2^{kd} D_{2^k}K \quad\text{for}\quad k\in\N
    $$
    there holds 
    $$
        \norm{f \sep \mathcal{E}^s_{u,p,q}(\R)} 
        \sim \norm{ \bigg( \sum_{k=0}^\infty 2^{ksq} \abs{(K_k \ast f)(\cdot)}^q \bigg)^{\frac{1}{q}} \sep \mathcal{M}^u_p(\R)}, \qquad f \in \mathcal{E}^s_{u,p,q}(\R),
    $$
    with constants independent of $f$ and the usual modification if $q=\infty$.
\end{prop}

Moreover, we have the following Gagliardo-Nierenberg inequality.
\begin{prop}\label{lem:gagliardo}
    For $i=0,1$ let $0<p_i\leq u_i<\infty$ and $0< q_i, r\leq\infty$ as well as $s_i\in\mathbb{R}$, where $s_0<s_1$. 
    For $0<\Theta<1$ define
    $$
        s_0 < s_\Theta:=(1-\Theta)\, s_0 + \Theta\, s_1 < s_1
        \quad\text{ and }\quad
        \frac{1}{u_\Theta} := \frac{1-\Theta}{u_0} + \frac{\Theta}{u_1} \leq \frac{1-\Theta}{p_0} + \frac{\Theta}{p_1}
        =: \frac{1}{p_\Theta}.
    $$
    Then for all $f\in\mathcal{S}'(\R)$ there holds (with a constant independent of $f$)
    $$
        \norm{f \sep \mathcal{E}^{s_\Theta}_{u_\Theta,p_\Theta,r}(\R)} 
        \lesssim \norm{f \sep \mathcal{E}^{s_0}_{u_0,p_0,q_0}(\R)}^{1-\Theta} \norm{f \sep \mathcal{E}^{s_1}_{u_1,p_1,q_1}(\R)}^{\Theta}.
    $$
\end{prop}
\begin{proof}
    Since $\mathcal{E}^s_{u,p,q}(\R)=F^{s,\frac{1}{p}-\frac{1}{u}}_{p,q}(\R)$ we can apply \cite[Prop.~3.1]{Sic13} with $\tau_i:=\frac{1}{p_i}-\frac{1}{u_i}\geq 0$, $i=0,1$, such that $\tau_\Theta :=(1-\Theta)\,\tau_0 + \Theta\,\tau_1=\frac{1}{p}-\frac{1}{u}\geq 0$.
\end{proof}

Finally, for some parameters, all tempered distributions from a Triebel-Lizorkin-Morrey space can be interpreted as locally $p$-integrable \emph{functions} in $\mathcal{M}^u_p(\R)$.
\begin{prop}\label{prop:E_in_M}
    Let $0<p\leq u < \infty$ and $0<q\leq\infty$. Then
    \begin{enumerate}
        \item $p>1$ implies $\mathcal{E}^{0}_{u,p,2}(\mathbb{R}^{d}) = \mathcal{M}^{u}_{p}(\R)$.

        \item $q\leq 2$ if $p\geq 1$ or $q<\infty$ if $p<1$ shows $\mathcal{E}^{\frac{p}{u}\sigma_p}_{u,p,q}(\R) \hookrightarrow \mathcal{M}^u_p(\R)$.
    \end{enumerate}
    In particular, for $u\geq 1$ there holds $\mathcal{E}^{0}_{u,\max\{1,p\},2}(\mathbb{R}^{d}) \hookrightarrow \mathcal{M}^{u}_{p}(\R)$.
\end{prop}

\begin{remark}\label{rem:E_in_M}
    The essence of the proof of (ii) for $p<1$ below is that for $f\in \mathcal{E}^{s}_{u,p,q}(\R)$ with the given choice of parameters, the sequence $\sum_{k=0}^N \mathcal{F}^{-1}(\varphi_k\,\mathcal{F}f)$ converges in $\mathcal{S}'(\R)$ \emph{and} in $\mathcal{M}^u_p(\R)$ to the \emph{same limit} (namely $f$), as $N\to\infty$, which is not necessarily the case in general. 
    In this regard, we refer to the corresponding discussion in the context of Besov spaces and $L_p(\R)$ in \cite[Lem.~2.2.4]{RuSi96} as well as \cite[Rem.~2.3.2/2]{Tr92} and \cite[Rem.~2.5.3/1-3]{Tr83}.
\end{remark}

\begin{proof}[Proof (of \autoref{prop:E_in_M})]
    \emph{Step 1. } Part~(i) has been shown, e.g., in \cite[Prop.~4.1]{maz}. 
    
    \emph{Step 2. } The case $p\geq 1$ in (ii) is covered by \cite[Thm.~3.2]{HaMoSkr2020}.     
    Therefore, we focus on $p<1$ such that $s^\ast:=\frac{p}{u}\sigma_p>0$. 
    
    Given $f \in \mathcal{S}'(\R)$ and $N\in\N_0$ let us define $f^N:=\sum_{k=0}^N u_k$, where $u_k:=\mathcal{F}^{-1}(\varphi_k\,\mathcal{F}f)$, $k\in\N_0$.
    Clearly, each $u_k$ and hence all $f^N$ belong to $\mathcal{S}'(\R)$, but can be identified with entire analytic (and hence measurable) functions.
    Moreover, in $\mathcal{S}'(\R)$ there holds $f^N \stackrel{N\to\infty}{\longrightarrow} f$.
    If we additionally assume that $f \in \mathcal{E}^{s}_{u,p,q}(\R) \hookrightarrow \mathcal{S}'(\R)$ with $s\in\mathbb{R}$, then the dyadic annuli criterion (\autoref{prop:dyadic_crit}) yields that $f^N$ converges in $\mathcal{S}'(\R)$ to some $U \in \mathcal{E}^{s}_{u,p,q}(\R)$, where (as limits in $\mathcal{S}'(\R)$ are unique) we have $U=f$. If we further assume that $q<\infty$, \autoref{prop:dyadic_crit} even yields that $f^N \stackrel{N\to\infty}{\longrightarrow} f$ in $\mathcal{E}^{s}_{u,p,q}(\R)$.
    
    On the other hand, under the same assumptions (i.e.\ $f \in \mathcal{E}^{s}_{u,p,q}(\R)$ with $s\in\mathbb{R}$ and $q<\infty$), we can estimate
    $$
        \abs{u_k(x)} = 2^{-ks} \abs{2^{ks} \, u_k(x)} \leq 2^{-ks} \bigg(\sum_{j=0}^\infty \abs{2^{js} \, u_j(x)}^q \bigg)^{\frac{1}{q}}, \qquad x\in\R,\; k\in\N_0,
    $$
    such that \autoref{lem:M}(i) yields $u_k\in \mathcal{M}^u_p(\R)$ with $\norm{u_k\sep \mathcal{M}^u_p(\R)} \leq 2^{-ks} \norm{f \sep \mathcal{E}^{s}_{u,p,q}(\R)}$ and thus also $f^N\in \mathcal{M}^u_p(\R)$ for every $N\in\N_0$. 
    If even $s>0$, then $(f^N)_{N\in\N_0}$ forms a Cauchy sequence in $\mathcal{M}^u_p(\R)$ as \autoref{lem:M}(iv) and the previous estimate imply that
    \begin{align*}
        \norm{f^M - f^N \sep \mathcal{M}^u_p(\R)}^p 
        &= \norm{\sum_{k=N+1}^M u_k \sep \mathcal{M}^u_p(\R)}^p \\
        &\leq \sum_{k=N+1}^M \norm{ u_k \sep \mathcal{M}^u_p(\R)}^p \\
        &\leq \sum_{k=N+1}^\infty 2^{-ksp} \norm{f \sep \mathcal{E}^{s}_{u,p,q}(\R)}^p < \varepsilon, \qquad M>N\geq N_0(\varepsilon).
    \end{align*}
    By completeness, in $\mathcal{M}^u_p(\R)$ it converges to some $f^\infty := \lim_{N\to\infty}\limits f^N = \sum_{k=0}^\infty u_k \in \mathcal{M}^u_p(\R)$ for which we similarly find that
    $$
        \norm{f^\infty \sep \mathcal{M}^u_p(\R)}^p 
        \leq \sum_{k=0}^\infty \norm{ u_k \sep \mathcal{M}^u_p(\R)}^p \\
        \leq \sum_{k=0}^\infty 2^{-ksp} \norm{f \sep \mathcal{E}^{s}_{u,p,q}(\R)}^p
        \sim \norm{f \sep \mathcal{E}^{s}_{u,p,q}(\R)}^p,
    $$
    as $s>0$.  
    
    Hence, to prove (ii) for $p<1$, it therefore suffices to show that under the additional condition $s=s^\ast=\frac{p}{u}\,\sigma_p$ every $f \in \mathcal{E}^{s}_{u,p,q}(\R)$ can be interpreted as a function in $L_p^\loc(\R)$ for which $f=f^\infty$ almost everywhere.
    To see this, we note that from \cite[Thm.~3.2]{HaMoSkr2020} and \eqref{eq:morrey_embedding} with $p<1$ it follows $\mathcal{E}^{s^\ast}_{u,p,q}(\R) \hookrightarrow \mathcal{M}^{\frac{u}{p}}_{1}(\R)\hookrightarrow \mathcal{M}^{\frac{u}{p}}_{p}(\R)$ such that indeed $f\in L_p^\loc(\R)$.
    Further, it is easily seen that the expressions
    $$
        (g,h) \mapsto d_\ell(g,h) := \norm{g-h \sep L_p(B_\ell)}^p,
        \qquad\text{where}\quad B_\ell:=B(0,\ell)\subseteq\R \quad\text{and}\quad \ell\in\N,
    $$
    define a countable family of pseudo metrics on $L_p^\loc(\R)$ which is separating in the sense that $g,h\in L_p^\loc(\R)$ with $d_\ell(g,h)=0$ for all $\ell\in\N$ implies $g=h$ almost everywhere.    
    Next we note that for $g,h\in \mathcal{M}^{\widetilde{u}}_{p}(\R)$ with $0<p \leq \min\{1,\widetilde{u}\}$ there holds
    $$
        d_\ell(g,h) = \Big( \abs{B_\ell}^{\frac{1}{\widetilde{u}}-\frac{1}{p}} \norm{g-h \sep L_p(B_\ell)} \Big)^p \abs{B_\ell}^{1-\frac{p}{\widetilde{u}}}
        \lesssim \norm{g-h \sep \mathcal{M}^{\widetilde{u}}_{p}(\R)}^p \, \ell^{d(1-\frac{p}{\widetilde{u}})}, \qquad \ell\in\N.
    $$
    Therefore, the Fr\'echet combination
    $$
        d(g,h) := \sum_{\ell=1}^\infty 2^{-\ell} \, \frac{d_\ell(g,h)}{1 + d_\ell(g,h)}
    $$
    defines a metric on $L_p^\loc(\R)$ which satisfies
    $$
        d(g,h) \lesssim \sum_{\ell=1}^\infty \frac{\ell^{d(1-\frac{p}{\widetilde{u}})}}{2^{\ell}} \, \frac{\norm{g-h \sep \mathcal{M}^{\widetilde{u}}_{p}(\R)}^p}{1 + d_\ell(g,h)}
        \lesssim \norm{g-h \sep \mathcal{M}^{\widetilde{u}}_{p}(\R)}^p,
        \qquad g,h\in \mathcal{M}^{\widetilde{u}}_{p}(\R).
    $$
    So, based on the arguments above we find that
    \begin{align*}
        d(f^\infty,f) 
        &\leq d(f^\infty,f^N) + d(f^N,f) \\
        &\lesssim \norm{f^\infty-f^N \sep \mathcal{M}^{u}_{p}(\R)}^p + \norm{f^N-f \sep \mathcal{M}^{\frac{u}{p}}_{p}(\R)}^p \\
        &\lesssim \norm{f^\infty-f^N \sep \mathcal{M}^{u}_{p}(\R)}^p + \norm{f^N-f \sep \mathcal{E}^{s^\ast}_{u,p,q}(\R)}^p
        \stackrel{N\to\infty}{\longrightarrow} 0
    \end{align*}
    and thus $f=f^\infty$ in $L_p^\loc(\R)$ as claimed.

    \emph{Step 3. }
    Finally, to prove the last statement, we distinguish two cases. If $p>1$, assertion~(i) shows $\mathcal{E}^{0}_{u,\max\{1,p\},2}(\R) = \mathcal{E}^{0}_{u,p,2}(\R) = \mathcal{M}^u_p(\R)$, while for $p\leq 1$ there holds
    $$
        \mathcal{E}^{0}_{u,\max\{1,p\},2}(\R) 
        = \mathcal{E}^{0}_{u,1,2}(\R)
        \hookrightarrow \mathcal{M}^u_1(\R)
        \hookrightarrow \mathcal{M}^u_p(\R)
    $$
    using (ii) with $p=1$ and \eqref{eq:morrey_embedding}.    
\end{proof}

\section{Bounds for the (Quasi-)Norm of Dilation Operators \texorpdfstring{$D_\lambda$}{Dlambda} on Triebel-Lizorkin-Morrey Spaces \texorpdfstring{$\mathcal{E}^s_{u,p,q}(\R)$}{EsupqRd}}
\label{sec_main_sec_main}
Here we prove the main results of this paper.
To start with, let us investigate (quasi-)norm estimates related to dilations by factors $\lambda$ near $1$.

\subsection{The Case \texorpdfstring{$\lambda \approx 1$}{lambda near 1}}\label{sec_dilation_near1}
\begin{prop}\label{prop:T_mu}
    Let $0 < p \leq u < \infty$, $0 < q \leq \infty$ and $s\in \mathbb{R}$ as well as $2^{-1}<\mu< 2$ and $f\in \mathcal{S}'(\mathbb{R}^d)$. 
    Then $D_\mu f\in \mathcal{E}^s_{u,p,q}(\mathbb{R}^d)$ if and only if $f \in \mathcal{E}^s_{u,p,q}(\mathbb{R}^d)$. 
    In this case there holds
    $$
        \norm{D_\mu f\sep \mathcal{E}^s_{u,p,q}(\mathbb{R}^d)} \sim \norm{f\sep \mathcal{E}^s_{u,p,q}(\mathbb{R}^d)}
    $$
    with implied constants independent of $f$ and $\mu$. In particular,
    $
        \norm{D_\mu \sep \mathcal{L}\big(\mathcal{E}^s_{u,p,q}(\R)\big)}
        \sim 1.
    $
\end{prop}
\begin{proof}
    W.l.o.g.\ we assume $q<\infty$, since in the case $q = \infty$ the usual modifications can be made. Further, it is enough to show that $f \in \mathcal{E}^s_{u,p,q}(\mathbb{R}^d)$ implies $D_\mu f\in \mathcal{E}^s_{u,p,q}(\mathbb{R}^d)$ with
    $$
        \norm{D_\mu f\sep \mathcal{E}^s_{u,p,q}(\mathbb{R}^d)} 
        \lesssim \norm{f\sep \mathcal{E}^s_{u,p,q}(\mathbb{R}^d)}.
    $$
    To this end, note that with $f$ also $D_\mu f$ belongs to $\mathcal{S}'(\mathbb{R}^d)$. 
    Setting $u_k:=\mathcal{F}^{-1}\big( \varphi_k \, \mathcal{F}[D_\mu f]\big)$ for $k\in\N_0$ and $U^N:=\sum_{k=0}^N u_k$ for $N\in\N$, we can decompose $D_\mu f$ into
    $$
        D_\mu f 
        = \sum_{k=0}^\infty u_k
        = \lim_{N\to\infty} U^N
        \qquad\text{(convergence in $\mathcal{S}'(\mathbb{R}^d)$)}.
    $$
    If we can now prove that 
    $$
        U^N \in \mathcal{E}^s_{u,p,q}(\mathbb{R}^d) \quad\text{with}\quad
        \norm{U^N\sep \mathcal{E}^s_{u,p,q}(\mathbb{R}^d)} 
        \lesssim \norm{f\sep \mathcal{E}^s_{u,p,q}(\mathbb{R}^d)},
        \qquad N\in\N,
    $$
    where the implied constant does not depend on $N$, then the Fatou property \cite[Prop.~2.8]{ysy} implies that also $D_\mu f\in \mathcal{E}^s_{u,p,q}(\R)$ with
    $$
        \norm{D_\mu f \sep \mathcal{E}^s_{u,p,q}(\R)} \leq \sup_{N\in\N} \norm{U^N \sep \mathcal{E}^s_{u,p,q}(\R)} \lesssim \norm{f \sep \mathcal{E}^s_{u,p,q}(\R)}.
    $$
    For that purpose, we shall employ the dyadic annuli criterion (\autoref{prop:dyadic_crit}) for the sequence $(\widetilde{u}_k^N)_{k\in\N_0}$ given by $\widetilde{u}^N_k:=u_k$ for $k\in\{0,\ldots,N\}$ and $\widetilde{u}^N_k:=0$ otherwise, as then
    $$
        U^N = \sum_{k=0}^N u_k = \sum_{k=0}^\infty \widetilde{u}^N_k, \qquad N\in\N_0.
    $$
    \autoref{prop:dyadic_crit} is applicable, since for $k=0,\ldots,N$,  Formula~\eqref{eq:building_block_shift} implies
    $$
        \widetilde{u}^N_k = u_k 
        = \mathcal{F}^{-1}\big( \varphi_k \, \mathcal{F}[D_\mu f]\big)
        = D_\mu \big[ \mathcal{F}^{-1} \big( [D_\mu \varphi_k]\, \mathcal{F}f\big)\big]
    $$
    and we obviously have $\supp (\mathcal{F} \widetilde{u}^N_0) \subseteq B(0,4)$ while $\supp (\mathcal{F} \widetilde{u}^N_k) \subseteq B(0, 2^{k+2}) \setminus B(0,2^{k-3})$ if $k>0$.
    So, it remains to bound
    \begin{align*}
        A^N 
        &:= \norm{ \bigg( \sum_{k=0}^\infty 2^{ksq} \abs{\widetilde{u}^N_k (\cdot) }^q \bigg)^{\frac{1}{q}} \sep \mathcal{M}^u_p(\R)} \\
        &= \norm{ \bigg( \sum_{k=0}^N  2^{ksq} \abs{D_\mu \big[ \mathcal{F}^{-1} \big( [D_\mu \varphi_k]\, \mathcal{F}f\big)\big] (\cdot) }^q \bigg)^{\frac{1}{q}} \sep \mathcal{M}^u_p(\R)} \\
        &= \norm{ D_\mu \bigg( \sum_{k=0}^N \abs{ \big[ \mathcal{F}^{-1} \big( [D_\mu \varphi_k]\, 2^{ks} \mathcal{F}f\big)\big] (\cdot) }^q \bigg)^{\frac{1}{q}} \sep \mathcal{M}^u_p(\R)} \\
        &\sim \norm{ \bigg( \sum_{k=0}^N \abs{ \big[ \mathcal{F}^{-1} \big( [D_\mu \varphi_k]\, 2^{ks} \mathcal{F}f\big)\big] (\cdot) }^q \bigg)^{\frac{1}{q}} \sep \mathcal{M}^u_p(\R)}
    \end{align*}
    from above by $\norm{f \sep \mathcal{E}^s_{u,p,q}(\R)}$; see \autoref{lem:M}(iii) for the last step.
    To show this, we let $N\in\N$ and note that from $b\leq 2\leq 2 a$ and $2^{-1}<\mu<2$ it follows 
    $$
        \supp(D_\mu\varphi_k) 
        \subseteq \left\{x\in\R : \frac{2^{k-1}a}{\mu} \leq \abs{x} \leq \frac{2^k b}{\mu} \right\}
        \subseteq \left\{x\in\R : 2^{k-3}b \leq \abs{x} \leq 2^{k+2}a \right\}=:R_k
    $$
    if $k\in\N$, while $\supp(D_\mu\varphi_0) \subseteq \left\{x\in\R : \abs{x}\leq \frac{b}{\mu}\right\} \subseteq B(0, 4a) =: R_0$.
    Furthermore (for this proof) let us define $\varphi_{-2} := \varphi_{-1} :\equiv 0$. 
    Then for $k\in\N_0$ there holds $\widetilde{\varphi}_k := \sum_{j=k-2}^{k+2} \varphi_j \equiv 1$ on $R_k$ such that we can write
    $$
        \mathcal{F}^{-1} \big( [D_\mu \varphi_k]\, 2^{ks} \mathcal{F}f\big)
        = \mathcal{F}^{-1} \Big( [D_\mu \varphi_k]\, \mathcal{F} \big( \mathcal{F}^{-1} [\widetilde{\varphi}_k\, 2^{ks}\, \mathcal{F}f] \big)\Big)
        = \mathcal{F}^{-1} \big( M_k\, \mathcal{F}f_k \big),
    $$
    where $M_k := D_\mu\varphi_k$ and $f_k := \mathcal{F}^{-1} [ \widetilde{\varphi}_k \, 2^{ks} \,\mathcal{F}f]$ satisfies $\supp(\mathcal{F}f_k)\subseteq \Omega_k := B(0, d_k)$ with $d_k:=2^{k+3}$. 
    Moreover, recall that for $k\in\N$ we have $\varphi_k = D_{2^{-k}}\psi$ such that the Bessel potential norms of
    $$
         D_{d_k} M_k = D_{2^{k+3}} D_\mu D_{2^{-k}}\psi = D_{8\mu}\psi \in \mathcal{D}(\R)
         \qquad\text{and}\qquad
         D_{d_0} M_0 = D_{8\mu}\varphi_0 \in \mathcal{D}(\R)
    $$
    do not depend on $k$. 
    Hence, for $\nu > \frac{d}{2}+\frac{d}{\min\{p,q\}}$
    $$
        \sup_{k\in\N_0} \norm{D_{d_k}M_k \sep H^\nu(\R)}
        \leq \sup_{4\leq \lambda \leq 16} \max\left\{\norm{D_{\lambda}\psi \sep H^\nu(\R)}, \norm{D_\lambda \varphi_0 \sep H^\nu(\R)}\right\}
    $$
    is finite and the vector-valued Fourier multiplier assertion given in \autoref{prop:vv-FM} yields
    \begin{align*}
        A^N
        &\sim \norm{ \bigg( \sum_{k=0}^N \abs{ \big( \mathcal{F}^{-1} [ M_k\, \mathcal{F}f_k ] \big) (\cdot) }^q \bigg)^{\frac{1}{q}} \sep \mathcal{M}^u_p(\R)} 
        &\lesssim \norm{ \bigg( \sum_{k=0}^N \abs{ f_k (\cdot) }^q \bigg)^{\frac{1}{q}} \sep \mathcal{M}^u_p(\R)}.
    \end{align*}
    Now for each $k\in\N_0$ and a.e.\ $x\in\R$ we have
    $$
        \abs{f_k(x)} 
        \leq 2^{2s} \sum_{j=k-2}^{k+2} \abs{\big(\mathcal{F}^{-1} [ \varphi_j \, 2^{(k-2)s} \mathcal{F}f ] \big)(x)}
        \lesssim  \sum_{j=k-2}^{k+2} 2^{js} \abs{\big(\mathcal{F}^{-1} [ \varphi_j \, \mathcal{F}f] \big)(x)}
    $$
    and thus
    $$
        \bigg( \sum_{k=0}^N \abs{ f_k (x) }^q \bigg)^{\frac{1}{q}} 
        \lesssim \bigg( \sum_{j=0}^{N+2}2^{jsq} \abs{\big(\mathcal{F}^{-1} [ \varphi_j \, \mathcal{F}f] \big)(x) }^q \bigg)^{\frac{1}{q}}
        \leq \bigg( \sum_{j=0}^{\infty} 2^{jsq} \abs{\big(\mathcal{F}^{-1} [ \varphi_j \, \mathcal{F}f] \big)(x) }^q \bigg)^{\frac{1}{q}}
    $$
    with an implied constant independent of $N\in\N_0$. 
    In view of the previous estimates and \autoref{lem:M}(i) this finally yields
    $$
        \norm{U^N \sep \mathcal{E}^s_{u,p,q}(\R)} 
        \lesssim A^N
        \lesssim \norm{ \bigg( \sum_{j=0}^{\infty} 2^{jsq} \abs{\big(\mathcal{F}^{-1} [ \varphi_j \, \mathcal{F}f] \big)(\cdot) }^q \bigg)^{\frac{1}{q}} \sep \mathcal{M}^u_p(\R)}
        = \norm{f \sep \mathcal{E}^s_{u,p,q}(\R)}
    $$
    and the proof is complete.
\end{proof}

Let us stress that the same proof technique can be used to show
that for all $\lambda > 1$ there exists $c_\lambda>0$ such that
$$
        \norm{D_\lambda f\sep \mathcal{E}^s_{u,p,q}(\mathbb{R}^d)} 
        \leq c_\lambda \norm{f\sep \mathcal{E}^s_{u,p,q}(\mathbb{R}^d)}, \qquad f\in \mathcal{E}^s_{u,p,q}(\mathbb{R}^d).
$$
However, due to the use of \autoref{prop:vv-FM},
this constant $c_\lambda$ would depend on $\lambda$ in a sub-optimal way.
Therefore, we shall instead employ \autoref{prop:FM_M} in order to prove our general upper bounds in the following subsection; see \autoref{thm:upper_bounds} below.

\begin{remark}\label{rem:reduction}
    Note that \autoref{prop:T_mu} allows to reduce (quasi-)norm estimates of the form $\lambda^\alpha \, (\log_2 \lambda)^\beta$ for dilation operators $D_\lambda$, where $\lambda \geq 2$ and $\alpha,\beta\in\mathbb{R}$, to the dyadic case $\lambda=\lambda_j:=2^j$ with $j\in\N$.
    Indeed, for $\lambda > 2$ there exists an unique $j\in\N\setminus\{1\}$ such that
    $$
        \mu := \lambda_j^{-1}\, \lambda \in (2^{-1}, 1].
    $$
    Then we have $\frac{1}{2}\,\lambda_j < \lambda \leq \lambda_j$  and $\frac{1}{2} \, \log_2 \lambda_j = \frac{1}{2}\, j \leq j-1 < \log_2(2^j \, \mu) = \log_2\lambda \leq \log_2\lambda_j$
    which shows that
    $$
        \lambda_j^\alpha \, (\log_2 \lambda_j)^\beta 
        \sim \lambda^\alpha \, (\log_2 \lambda)^\beta,
        \qquad \alpha,\beta\in\mathbb{R}.
    $$
    Moreover, $\lambda_j = \lambda \, \mu^{-1}$
    and $\lambda=\lambda_{j}\,\mu$ imply
    $D_{\lambda_j} = D_{\lambda}\circ D_{\mu^{-1}}$ as well as $D_\lambda = D_{\lambda_j}\circ D_\mu$ such that \autoref{prop:T_mu} yields that for all $0<p\leq u < \infty$, $0<q\leq\infty$ and $s\in\mathbb{R}$
    \begin{align*}
        \norm{ D_{\lambda_j} \sep \mathcal{L}\big(\mathcal{E}^s_{u,p,q}(\R)\big) }
        &\leq \norm{ D_{\lambda} \sep \mathcal{L}\big(\mathcal{E}^s_{u,p,q}(\R)\big) } \norm{ D_{\mu^{-1}} \sep \mathcal{L}\big(\mathcal{E}^s_{u,p,q}(\R)\big) }  \\
        &\sim \norm{ D_{\lambda} \sep \mathcal{L}\big(\mathcal{E}^s_{u,p,q}(\R)\big) } \\
        &\lesssim \norm{ D_{\lambda_j} \sep \mathcal{L}\big(\mathcal{E}^s_{u,p,q}(\R)\big) },
    \end{align*}
    i.e., $\norm{ D_{\lambda_j} \sep \mathcal{L}\big(\mathcal{E}^s_{u,p,q}(\R)\big) } \sim \norm{ D_{\lambda} \sep \mathcal{L}\big(\mathcal{E}^s_{u,p,q}(\R)\big) }$.
\end{remark}

\subsection{Upper Bounds for \texorpdfstring{$\lambda\geq 2$}{lambda>=2}}\label{sec_upper_bounds_1}
In this subsection we prove the upper estimates for the operator (quasi-)norm $\norm{D_\lambda \sep \mathcal{L}\big(\mathcal{E}^s_{u,p,q}(\R)\big)}$, where $\lambda \geq 2$. 

\pagebreak
\begin{prop}\label{thm:upper_bounds}
    Let $0<p\leq u <\infty$, $0<q\leq \infty$ as well as $s\in\mathbb{R}$ and $\lambda \geq 2$. Then for the restriction of $D_\lambda$ to $\mathcal{E}^s_{u,p,q}(\R)$ there holds
    \begin{enumerate}
        \item $\norm{D_\lambda \sep \mathcal{L}\big(\mathcal{E}^s_{u,p,q}(\R)\big)} \lesssim \lambda^{s- \frac{d}{u}}$ if $s>\sigma_p$.

        \item $\norm{D_\lambda \sep \mathcal{L}\big(\mathcal{E}^{\sigma_p}_{u,p,q}(\R)\big)} \lesssim \lambda^{\sigma_p- \frac{d}{u}} \, (\log_2 \lambda)^{\max\{\frac{1}{p}, \frac{1}{q}\}}$.
        
        If we additionally assume that $p>1$ (such that $\sigma_p=0$), then we even have
        $\norm{D_\lambda \sep \mathcal{L}\big(\mathcal{E}^{\sigma_p}_{u,p,q}(\R)\big)} \lesssim \lambda^{\sigma_p- \frac{d}{u}} \, (\log_2 \lambda)^{\max\{0, \frac{1}{q} - \frac{1}{2}\}}$. 
        
        \item $\norm{D_\lambda \sep \mathcal{L}\big(\mathcal{E}^s_{u,p,q}(\R)\big)} \lesssim \lambda^{\sigma_p- \frac{d}{u}}$ if $s<\sigma_p$.
    \end{enumerate}
    Therein the implicit constants do not depend on $\lambda$. To incorporate the case $q = \infty$  we use the convention $\frac{1}{\infty}= 0$.
\end{prop}
\begin{proof}
    \emph{Step 1 (Preparations). } In view of \autoref{rem:reduction} it suffices to prove the claim for $\lambda=2^j$ with $j\in\N$. Further, w.l.o.g.\ we may assume $q<\infty$, since otherwise the usual modifications can be made.  Let $j\in\N$ and $f\in \mathcal{E}^s_{u,p,q}(\R)$ be fixed and recall that $\psi_\ell := D_{2^{-\ell}}\psi$ if $\ell\in\mathbb{Z}$ and $\varphi_k = \psi_k$ for $k\in\N$.
    Hence, $D_{2^j}\varphi_k = D_{2^{-(k-j)}}\psi = \psi_{k-j}$ for all $k\in\N$ and by \eqref{eq:building_block_shift} we have
    \begin{align*}
        u_k
        :=\mathcal{F}^{-1} \big( \varphi_k \, \mathcal{F} [D_{2^j} f]\big)
        &= D_{2^j} \big[\mathcal{F}^{-1} \big( [D_{2^j} \varphi_k] \, \mathcal{F} f \big) \big] \\
        &=\begin{cases}
            D_{2^j} \big[\mathcal{F}^{-1} \big( [D_{2^j} \varphi_0] \, \mathcal{F} f \big) \big], & k=0,\\
            D_{2^j} \big[\mathcal{F}^{-1} \big( [D_{2^{-(k-j)}}\psi] \, \mathcal{F} f \big) \big], & k=1,\ldots,j,\\
            D_{2^j} \big[\mathcal{F}^{-1} \big( \varphi_{k-j} \, \mathcal{F} f \big) \big], &k\geq j+1,
        \end{cases}
    \end{align*}
    where obviously $\supp(\mathcal{F}u_0) \subseteq B(0,4)$ and
    $$
        \supp(\mathcal{F}u_k) \subseteq B(0,2^{k+2}) \setminus B(0,2^{k-3}), \qquad k\in\N.
    $$
    Therefore, in $\mathcal{S}'(\R)$ we can decompose $D_{2^j}f$ into
    \begin{align*}
        D_{2^j}f 
        &= \sum_{k=0}^\infty \mathcal{F}^{-1} \big( \varphi_k \, \mathcal{F} [D_{2^j} f]\big) 
        = u_0 + \sum_{k=1}^j u_k + \sum_{k=j+1}^\infty u_k =: U_1 + U_2 + U_3
    \end{align*}
    with convergence of the series in $\mathcal{S}'(\R)$. 
    In the following steps, we shall employ the dyadic annuli criterion (\autoref{prop:dyadic_crit}) to show that $U_n\in \mathcal{E}^s_{u,p,q}(\R)$ for $n=1,2,3$, where
    \begin{align*}
        \norm{U_1 \sep \mathcal{E}^s_{u,p,q}(\R)} 
        &\lesssim 2^{j(\sigma_p-\frac{d}{u})} \norm{f \sep \mathcal{E}^s_{u,p,q}(\R)}, \\
        \norm{U_2 \sep \mathcal{E}^s_{u,p,q}(\R)} 
        &\lesssim \begin{cases}
            2^{j(s-\frac{d}{u})} \norm{f \sep \mathcal{E}^s_{u,p,q}(\R)}, & s>\sigma_p,\\
            2^{j(\sigma_p-\frac{d}{u})} j^{\max\{\frac{1}{p},\frac{1}{q}\}} \norm{f \sep \mathcal{E}^s_{u,p,q}(\R)}, & s=\sigma_p, \\
            2^{j(\sigma_p-\frac{d}{u})} j^{\max\{0,\frac{1}{q}-\frac{1}{2}\}} \norm{f \sep \mathcal{E}^s_{u,p,q}(\R)}, & s=\sigma_p \text{ and } p>1, \\
            2^{j(\sigma_p-\frac{d}{u})} \norm{f \sep \mathcal{E}^s_{u,p,q}(\R)}, & s < \sigma_p,
        \end{cases} \\
        \norm{U_3 \sep \mathcal{E}^s_{u,p,q}(\R)} 
        &\lesssim 2^{j(s-\frac{d}{u})} \norm{f \sep \mathcal{E}^s_{u,p,q}(\R)},
    \end{align*}
    such that the claim for $\lambda=2^j$ easily follows.

    \emph{Step 2 ($U_1$). } We apply \autoref{prop:dyadic_crit} to the sequence $(\widetilde{u}_k)_{k\in\N_0}$, where $\widetilde{u}_0 := u_0$ and $\widetilde{u}_k :\equiv 0$ for $k\in\N$ such that $U_1=\sum_{k=0}^\infty \widetilde{u}_k$ belongs to $\mathcal{E}^s_{u,p,q}(\R)$ provided that
    $$
        A_1 := \norm{\bigg( \sum_{k=0}^\infty 2^{ksq} \abs{\widetilde{u}_k(\cdot)}^q \bigg)^{\frac{1}{q}} \sep \mathcal{M}^u_p(\R)} 
        = \norm{u_0 \sep \mathcal{M}^u_p(\R)}
    $$
    is finite. In this case, $\norm{ U_1 \sep \mathcal{E}^s_{u,p,q}(\R)} \lesssim A_1$.
    The definition of $u_0$ and \autoref{lem:M}(iii) yield
    \begin{align*}
        A_1
        &= \norm{ D_{2^j} \big[\mathcal{F}^{-1} \big( [D_{2^j}\varphi_0] \, \mathcal{F} f\big) \big] \sep \mathcal{M}^u_p(\R)}
        \sim 2^{-j \frac{d}{u}} \norm{ \mathcal{F}^{-1} \big( [D_{2^j}\varphi_0] \, \mathcal{F} f\big) \sep \mathcal{M}^u_p(\R)}.
    \end{align*}
    Furthermore, we recall that $\supp \varphi_0 \subseteq B(0,2)$ with $\varphi_0(x)=1$ for $\abs{x}\leq 1$ such that we have $\supp (D_{2^j}\varphi_0) \subseteq B(0,2^{1-j}) \subseteq B(0,1)$ and therefore $D_{2^j}\varphi_0= [D_{2^j}\varphi_0] \, \varphi_0$ in $\mathcal{D}(\R)$.
    Setting $f_0 := \mathcal{F}^{-1} [\varphi_0\, \mathcal{F} f]$ we thus have 
    $$
        \mathcal{F}^{-1} \big( [D_{2^j}\varphi_0] \, \mathcal{F} f\big)
        = \mathcal{F}^{-1} \Big( [D_{2^j}\varphi_0] \, \mathcal{F} \big( \mathcal{F}^{-1} [\varphi_0\, \mathcal{F} f] \big) \Big) 
        = \mathcal{F}^{-1} \big( [D_{2^j}\varphi_0] \, \mathcal{F} f_0\big),
    $$
    where $\supp (\mathcal{F}f_0) \subseteq B(0,2)$.
    Hence, \autoref{prop:FM_M} and \eqref{eq:FTj_phi} give
    \begin{align}
        \norm{\mathcal{F}^{-1} \big( [D_{2^j}\varphi_0] \, \mathcal{F} f\big) \sep \mathcal{M}^u_p(\R)}
        &= \norm{\mathcal{F}^{-1} \big( [D_{2^j}\varphi_0] \, \mathcal{F} f_0\big) \sep \mathcal{M}^u_p(\R)} \nonumber\\
        &\lesssim (1+2)^{\sigma_p} \norm{\mathcal{F}^{-1} (D_{2^j}\varphi_0 ) \sep L_{\min\{1,p\}}(\R)} \norm{f_0 \sep \mathcal{M}^u_p(\R)} \nonumber\\
        &\sim 2^{-j d} \norm{D_{2^{-j}} (\mathcal{F}^{-1}\varphi_0) \sep L_{\min\{1,p\}}(\R)} \norm{f_0 \sep \mathcal{M}^u_p(\R)} \nonumber\\
        &= 2^{-jd+\frac{jd}{\min\{1,p\}}} \norm{\mathcal{F}^{-1}\varphi_0 \sep L_{\min\{1,p\}}(\R)} \norm{f_0 \sep \mathcal{M}^u_p(\R)} \nonumber\\
        &\sim 2^{j\sigma_p} \norm{f_0 \sep \mathcal{M}^u_p(\R)}, \label{eq:step1}
    \end{align}
    as for $0<r<\infty$ there holds $\norm{D_{2^{-j}}(\mathcal{F}^{-1}\varphi_0) \sep L_{r}(\R)} = 2^{\frac{jd}{r}} \norm{\mathcal{F}^{-1}\varphi_0 \sep L_{r}(\R)} <\infty$ and we have $d(-1 + \frac{1}{\min\{1,p\}})=\sigma_p$.
    In conclusion, we can employ \autoref{lem:M}(i) to obtain
    \begin{align*}
        \norm{ U_1 \sep \mathcal{E}^s_{u,p,q}(\R)} 
        &\lesssim A_1 \\
        &\lesssim 2^{j(\sigma_p-\frac{d}{u})} \norm{f_0 \sep \mathcal{M}^u_p(\R)} \\
        &\leq 2^{j(\sigma_p-\frac{d}{u})} \norm{\bigg( \sum_{k=0}^\infty 2^{ksq} \abs{ \big(\mathcal{F}^{-1}[ \varphi_k \, \mathcal{F} f] \big) (\cdot) }^q \bigg)^{\frac{1}{q}} \sep \mathcal{M}^u_p(\R)} \\
        &\sim 2^{j(\sigma_p-\frac{d}{u})} \norm{f \sep \mathcal{E}^s_{u,p,q}(\R)}.
    \end{align*}
    
    \emph{Step 3 ($U_3$). }
    In $\mathcal{S}'(\R)$ there holds
    $$
        U_3 = \sum_{k=j+1}^\infty u_k = \lim_{N\to\infty} U_3^N
        \qquad\text{(convergence in $\mathcal{S}'(\R)$)},
    $$
    where we set $U_3^N:=\sum_{k=j+1}^N u_k \in \mathcal{S}'(\R)$ for $N\in\N$ with $N>j$.
    If we can show that $U_3^N \in \mathcal{E}^s_{u,p,q}(\R)$ with (quasi-)norms bounded independent of $N$, then the Fatou property \cite[Prop.~2.8]{ysy} implies that also $U_3\in \mathcal{E}^s_{u,p,q}(\R)$ and
    $$
        \norm{U_3 \sep \mathcal{E}^s_{u,p,q}(\R)} \leq \sup_{N>j} \norm{U_3^N \sep \mathcal{E}^s_{u,p,q}(\R)}.
    $$
    To this end, for fixed $N>j$, we again use the dyadic annuli criterion \autoref{prop:dyadic_crit} in order to represent $U_3^N$ in $\mathcal{S}'(\R)$ by the sequence $(\widetilde{u}_k^N)_{k\in\N_0}$ with $\widetilde{u}_k^N := u_k$ if $k=j+1,\ldots,N$ and $\widetilde{u}_k^N :\equiv 0$ otherwise:
    $$
        U_3^N = \sum_{k=j+1}^N u_k = \sum_{k=0}^\infty \widetilde{u}_k^N.
    $$
    The definitions of $\widetilde{u}_k^N$ and $u_k$ yield
    \begin{align*}
        A_3^N
        &:= \norm{ \bigg( \sum_{k=0}^\infty 2^{ksq} \abs{ \widetilde{u}_k^N(\cdot) }^q \bigg)^{\frac{1}{q}} \sep \mathcal{M}^u_p(\R)} \\
        &= \norm{ \bigg( \sum_{k=j+1}^N 2^{ksq} \abs{ u_k(\cdot) }^q \bigg)^{\frac{1}{q}} \sep \mathcal{M}^u_p(\R)} \\
        &= 2^{js} \norm{ \bigg( \sum_{k=j+1}^N 2^{(k-j)sq} \abs{ \big( D_{2^j} \big[\mathcal{F}^{-1} ( \varphi_{k-j} \, \mathcal{F} f)  \big] \big) (\cdot) }^q \bigg)^{\frac{1}{q}} \sep \mathcal{M}^u_p(\R)} \\
        &= 2^{js} \norm{ D_{2^j}\bigg( \sum_{\ell=1}^{N-j} 2^{\ell sq} \abs{ \big(\mathcal{F}^{-1} [ \varphi_{\ell} \, \mathcal{F} f ] \big)(\cdot) }^q \bigg)^{\frac{1}{q}} \sep \mathcal{M}^u_p(\R)},
    \end{align*}
    where we used that the sums actually consist of only finitely many terms.   \autoref{lem:M}(iii) and (i) imply that
    \begin{align*}
        A_3^N 
        &\sim 2^{j(s - \frac{d}{u})} \norm{ \bigg( \sum_{\ell=1}^{N-j} 2^{\ell sq} \abs{ \big(\mathcal{F}^{-1} [ \varphi_{\ell} \, \mathcal{F} f] \big) (\cdot) }^q \bigg)^{\frac{1}{q}} \sep \mathcal{M}^u_p(\R)} \\
        &\leq 2^{j(s - \frac{d}{u})} \norm{ \bigg( \sum_{k=0}^\infty 2^{k sq} \abs{ \big( \mathcal{F}^{-1}[ \varphi_{k} \, \mathcal{F} f]\big) (\cdot) }^q \bigg)^{\frac{1}{q}} \sep \mathcal{M}^u_p(\R)} 
        = 2^{j(s - \frac{d}{u})} \norm{ f \sep \mathcal{E}^s_{u,p,q}(\R)}
    \end{align*}
    is bounded independently of $N$. Hence, by \autoref{prop:dyadic_crit} and the Fatou property, we can conclude that indeed $U_3 \in \mathcal{E}^s_{u,p,q}(\R)$ with
    \begin{align*}
        \norm{U_3 \sep \mathcal{E}^s_{u,p,q}(\R)} 
        &\leq \sup_{N>j} \norm{U_3^N \sep \mathcal{E}^s_{u,p,q}(\R)} 
        \lesssim \sup_{N>j} A_3^N 
        \lesssim 2^{j(s - \frac{d}{u})} \norm{ f \sep \mathcal{E}^s_{u,p,q}(\R)},
    \end{align*}
    as claimed.

    \emph{Step 4 ($U_2$). }
    In order to analyze the remaining component $U_2=\sum_{k=1}^j u_k$, we again employ \autoref{prop:dyadic_crit}, where this time $(\widetilde{u}_k)_{k\in\N_0}$ is given by $\widetilde{u}_k:=u_k$ for $k=1,\ldots,j$, and zero otherwise.
    Thus, we need to bound
    \begin{align*}
        A_2 
        &:= \norm{ \bigg( \sum_{k=0}^\infty 2^{ksq} \abs{ \widetilde{u}_k(\cdot) }^q \bigg)^{\frac{1}{q}} \sep \mathcal{M}^u_p(\R)} \\
        &= \norm{ \bigg( \sum_{k=1}^j 2^{ksq} \abs{ u_k(\cdot) }^q \bigg)^{\frac{1}{q}} \sep \mathcal{M}^u_p(\R)} \\
        &= 2^{js} \norm{ \bigg( \sum_{k=1}^j 2^{(k-j)sq} \abs{ \big( D_{2^j} \big[\mathcal{F}^{-1} \big( [D_{2^{-(k-j)}}\psi] \, \mathcal{F} f \big) \big] \big)(\cdot) }^q \bigg)^{\frac{1}{q}} \sep \mathcal{M}^u_p(\R)} \\
        &= 2^{js} \norm{ D_{2^j} \bigg( \sum_{k=1}^j 2^{(k-j)sq} \abs{ \big[ \mathcal{F}^{-1} \big( [D_{2^{-(k-j)}}\psi] \, \mathcal{F} f \big) \big] (\cdot) }^q \bigg)^{\frac{1}{q}} \sep \mathcal{M}^u_p(\R)}.
    \end{align*}
    Setting $\ell:=j-k$ and using \autoref{lem:M}(iii) we find
    \begin{align}
    \label{eq:proof_II}
        A_2 \sim 2^{j(s-\frac{d}{u})} \norm{ \bigg( \sum_{\ell=0}^{j-1} 2^{-\ell sq} \abs{ \big[ \mathcal{F}^{-1} \big( [D_{2^\ell}\psi] \, \mathcal{F} f \big) \big] (\cdot) }^q \bigg)^{\frac{1}{q}} \sep \mathcal{M}^u_p(\R)}.
    \end{align}
    However, note that we cannot just extend this sum and proceed as before. 
    Therefore, in the sequel, we distinguish two methods of proof.

    \emph{Substep 4.1 (General approach). } Let us fix $\mu := \min\{p,q\}$.
    Then $\frac{\mu}{q} \leq 1 \leq \frac{p}{\mu}$ so that \autoref{lem:M}(ii) implies
    \begin{align*}
        &\norm{ \bigg( \sum_{\ell=0}^{j-1} 2^{-\ell s q} \abs{ \big[ \mathcal{F}^{-1} \big( [D_{2^\ell}\psi] \, \mathcal{F} f\big) \big](\cdot) }^{q} \bigg)^{\frac{1}{q}} \sep \mathcal{M}^u_p(\R)}^\mu \\
        &\qquad = \norm{ \bigg( \sum_{\ell=0}^{j-1} 2^{-\ell sq} \abs{ \big[\mathcal{F}^{-1} \big( [D_{2^\ell}\psi] \, \mathcal{F} f\big)\big] (\cdot) }^{q} \bigg)^{\frac{\mu}{q}} \sep \mathcal{M}^\frac{u}{\mu}_\frac{p}{\mu}(\R)} \\
        &\qquad \leq \sum_{\ell=0}^{j-1} 2^{-\ell s\mu} \norm{ \abs{ \big[ \mathcal{F}^{-1} \big( [D_{2^\ell}\psi] \, \mathcal{F} f\big)\big] (\cdot) }^{\mu} \sep \mathcal{M}^\frac{u}{\mu}_\frac{p}{\mu}(\R)} \\
        &\qquad = \sum_{\ell=0}^{j-1} 2^{-\ell s \mu} \norm{ \mathcal{F}^{-1} \big( [D_{2^\ell}\psi] \, \mathcal{F} f\big) \sep \mathcal{M}^u_p(\R)}^\mu
    \end{align*}
    and thus
    \begin{align*}
        A_2
        \lesssim 2^{j(s-\frac{d}{u})} \bigg( \sum_{\ell=0}^{j-1} 2^{-\ell s\mu} \norm{ \mathcal{F}^{-1} \big( [D_{2^\ell}\psi] \, \mathcal{F} f\big) \sep \mathcal{M}^u_p(\R)}^\mu \bigg)^{\frac{1}{\mu}}.
    \end{align*}
    Next, we like to apply the Fourier multiplier assertion \autoref{prop:FM_M}. To this end, note that $\supp \psi = \supp \varphi_0 \subseteq B(0,2)$ and hence 
    $$
        \supp (D_{2^\ell} \psi) \subseteq B(0,2^{1-\ell})\subseteq B(0,2), \qquad \ell=0,\ldots,j-1.
    $$ 
    Moreover, for 
    $\widetilde{\varphi}_0 := D_{2^{-1}}\varphi_0 
    =\varphi_0 + D_{2^{-1}}\varphi_0 - \varphi_0 
    = \varphi_0+\varphi_1$ 
    we have $\widetilde{\varphi}_0(x)=1$ if $\abs{x}<2$, while $\supp \widetilde{\varphi}_0\subseteq B(0,4)$.
    Setting $\widetilde{f}_0 := \mathcal{F}^{-1} [\widetilde{\varphi}_0\, \mathcal{F} f] $ such that $\supp (\mathcal{F}\widetilde{f}_0)\subseteq B(0,4)$, similar to \eqref{eq:step1} we obtain
    \begin{align*}
        \norm{\mathcal{F}^{-1} \big( [D_{2^\ell}\psi] \, \mathcal{F} f\big) \sep \mathcal{M}^u_p(\R)}
        &= \norm{\mathcal{F}^{-1} \big( [D_{2^\ell}\psi] \, \mathcal{F} \widetilde{f}_0\big) \sep \mathcal{M}^u_p(\R)} \\
        &\lesssim (2+4)^{\sigma_p} \norm{\mathcal{F}^{-1} (D_{2^\ell}\psi ) \sep L_{\min\{1,p\}}(\R)} \norm{\widetilde{f}_0 \sep \mathcal{M}^u_p(\R)} \\
        &\sim 2^{\ell \sigma_p} \norm{\widetilde{f}_0 \sep \mathcal{M}^u_p(\R)}
    \end{align*}
    for all $\ell=0,\ldots, j-1$, where due to \autoref{lem:M}(i)
    \begin{align}
        \norm{\widetilde{f}_0 \sep \mathcal{M}^u_p(\R)} 
        &= \norm{\mathcal{F}^{-1} [(\varphi_0+\varphi_1)\, \mathcal{F} f] \sep \mathcal{M}^u_p(\R)} \nonumber\\
        &\lesssim \norm{\mathcal{F}^{-1} [\varphi_0\, \mathcal{F} f] \sep \mathcal{M}^u_p(\R)} + \norm{2^s \abs{\mathcal{F}^{-1} [\varphi_1\, \mathcal{F} f] } \sep \mathcal{M}^u_p(\R)} \nonumber\\
        &\lesssim \norm{f \sep \mathcal{E}^s_{u,p,q}(\R)}. \label{eq:proof_f0}
    \end{align}
    Combining these estimates we conclude that in general $U_2 \in \mathcal{E}^s_{u,p,q}(\R)$ with 
    \begin{align}\label{eq:proof_A2}
        \norm{U_2 \sep \mathcal{E}^s_{u,p,q}(\R)} 
        \lesssim A_2 
        \lesssim 2^{j(s-\frac{d}{u})} \bigg( \sum_{\ell=0}^{j-1} 2^{-\ell (s-\sigma_p)\mu} \bigg)^{\frac{1}{\mu}} \norm{f \sep \mathcal{E}^s_{u,p,q}(\R)}.
    \end{align}
    If $s>\sigma_p$, then we can estimate
    $$
        \bigg( \sum_{\ell=0}^{j-1} 2^{-\ell (s-\sigma_p)\mu} \bigg)^{\frac{1}{\mu}}
        \leq \bigg( \sum_{\ell=0}^{\infty} \big( 2^{-(s-\sigma_p)\mu} \big)^{\ell} \bigg)^{\frac{1}{\mu}} < \infty
    $$
    independently of $j$ such that $\norm{U_2 \sep \mathcal{E}^s_{u,p,q}(\R)} \lesssim 2^{j(s-\frac{d}{u})} \norm{f \sep \mathcal{E}^s_{u,p,q}(\R)}$ as claimed.
    In case $s=\sigma_p$, we derive
    $$
        \bigg( \sum_{\ell=0}^{j-1} 2^{-\ell (s-\sigma_p)\mu} \bigg)^{\frac{1}{\mu}}
        = j^{\frac{1}{\mu}} = j^{\max\{\frac{1}{p}, \frac{1}{q}\}}
    $$
    and hence $\norm{U_2 \sep \mathcal{E}^{\sigma_p}_{u,p,q}(\R)} \lesssim 2^{j(\sigma_p-\frac{d}{u})} j^{\max\{\frac{1}{p}, \frac{1}{q}\}} \norm{f \sep \mathcal{E}^{\sigma_p}_{u,p,q}(\R)}$. And if $s<\sigma_p$, we obtain
    $$
        \bigg( \sum_{\ell=0}^{j-1} 2^{-\ell (s-\sigma_p)\mu} \bigg)^{\frac{1}{\mu}}
        = \bigg( \sum_{\ell=0}^{j-1} \big( 2^{(\sigma_p-s)\mu} \big)^{\ell} \bigg)^{\frac{1}{\mu}}
        \lesssim  \bigg( \big( 2^{(\sigma_p-s)\mu} \big)^{j} \bigg)^{\frac{1}{\mu}} 
        = 2^{j(\sigma_p-s)}
    $$
    such that $\norm{U_2 \sep \mathcal{E}^s_{u,p,q}(\R)} \lesssim 2^{j(\sigma_p-\frac{d}{u})} \norm{f \sep \mathcal{E}^s_{u,p,q}(\R)}$.

    \emph{Substep 4.2 (Refinement if $p>1$ and $s=\sigma_p=0$). } 
    Let us return to \eqref{eq:proof_II} under the additional assumption that $s=\sigma_p=0$, i.e.
    \begin{align*}
        A_2 
        \sim 2^{j(\sigma_p-\frac{d}{u})} \norm{ \bigg( \sum_{\ell=0}^{j-1} \abs{ \big[ \mathcal{F}^{-1} \big( [D_{2^\ell}\psi] \, \mathcal{F} f \big)\big] (\cdot) }^q \bigg)^{\frac{1}{q}} \sep \mathcal{M}^u_p(\R)}.
    \end{align*}
    As before, we can replace $f$ by $\widetilde{f}_0 = \mathcal{F}^{-1} [\widetilde{\varphi}_0\, \mathcal{F} f] $ since $\widetilde{\varphi}_0=\varphi_0+\varphi_1 \equiv 1$ on $\supp(D_{2^\ell}\psi)$.
    Next, if $q<2$, H\"older's inequality with $r:=\frac{2}{q}$ (i.e.\ $\frac{1}{r}=\frac{q}{2}$ and $\frac{1}{r'} = 1-\frac{q}{2} = q \max\{ 0, \frac{1}{q}-\frac{1}{2}\}$) yields that for a.e.\ $x\in\R$ there holds
    $$
        \bigg( \sum_{\ell=0}^{j-1} \abs{ \big[ \mathcal{F}^{-1} \big( [D_{2^\ell}\psi] \, \mathcal{F} \widetilde{f}_0 \big)\big] (x) }^q \bigg)^{\frac{1}{q}}
        \leq j^{\max\{ 0, \frac{1}{q}-\frac{1}{2}\}} \bigg( \sum_{\ell=0}^{j-1} \abs{ \big[\mathcal{F}^{-1} \big( [D_{2^\ell}\psi] \, \mathcal{F} \widetilde{f}_0 \big)\big] (x) }^2 \bigg)^{\frac{1}{2}}.
    $$
    If otherwise $q\geq 2$, then the same estimate holds true since $\frac{2}{q}\leq 1$ and $\max\{ 0, \frac{1}{q}-\frac{1}{2}\}=0$.
    Therefore, in both cases we obtain
    \begin{align*}
        \norm{U_2 \sep \mathcal{E}^{\sigma_p}_{u,p,q}(\R)} 
        &\lesssim A_2 \\
        &\lesssim 2^{j(\sigma_p-\frac{d}{u})} j^{\max\{ 0, \frac{1}{q}-\frac{1}{2}\}} \norm{ \bigg( \sum_{\ell=0}^{j-1} \abs{ \big[ \mathcal{F}^{-1} \big( [D_{2^\ell}\psi] \, \mathcal{F} \widetilde{f}_0 \big)\big] (\cdot) }^2 \bigg)^{\frac{1}{2}} \sep \mathcal{M}^u_p(\R)} \\
        &\leq 2^{j(\sigma_p-\frac{d}{u})} j^{\max\{ 0, \frac{1}{q}-\frac{1}{2}\}} \norm{ \bigg( \sum_{\ell=-\infty}^{\infty} \abs{ \big[ \mathcal{F}^{-1} \big( [D_{2^\ell}\psi] \, \mathcal{F} \widetilde{f}_0 \big)\big] (\cdot) }^2 \bigg)^{\frac{1}{2}} \sep \mathcal{M}^u_p(\R)}\\
        &\sim 2^{j(\sigma_p-\frac{d}{u})} j^{\max\{ 0, \frac{1}{q}-\frac{1}{2}\}} \norm{\widetilde{f}_0 \sep \mathcal{M}^u_p(\R)} \\
        &\lesssim 2^{j(\sigma_p-\frac{d}{u})} j^{\max\{ 0, \frac{1}{q}-\frac{1}{2}\}} \norm{f \sep \mathcal{E}^{\sigma_p}_{u,p,q}(\R)}, 
    \end{align*}
    where we used \autoref{prop:LittlewoodPaley_M} and \eqref{eq:proof_f0} (with $s=\sigma_p=0$) for the last two steps.
\end{proof}

\subsection{Lower Bounds for \texorpdfstring{$\lambda\geq 2$}{lambda>=2}}\label{sec_lower_bounds_1}
Here we show the corresponding estimates of $\norm{D_{\lambda} \sep \mathcal{L}\big(\mathcal{E}^s_{u,p,q}(\R)\big)}$ from below.

\begin{prop}\label{prop:lower_bounds}
    Let $0<p\leq u < \infty$, $0<q\leq \infty$ and $s\in\mathbb{R}$. 
    Then for $\lambda\geq 2$ we have
    $$
        \norm{D_{\lambda} \sep \mathcal{L}\big(\mathcal{E}^s_{u,p,q}(\R)\big)} \gtrsim \lambda^{\max\{s,0\} - \frac{d}{u}}
    $$
    and
    $$
    \norm{D_\lambda \sep \mathcal{L}\big( \mathcal{E}^0_{u,p,q}(\R) \big)} 
    \gtrsim \lambda^{-\frac{d}{u}} \cdot \begin{cases}
        \displaystyle (\log_2 \lambda)^{\frac{1}{q}-1}, & 0<u < 1,\\
        \displaystyle (\log_2 \lambda)^{\frac{1}{q}-\frac{1}{2}}, & 1\leq u < \infty,
    \end{cases}
    $$
    with implied constants that do not depend on $\lambda$. To incorporate the case $q = \infty$ we use the convention $\frac{1}{\infty}= 0$.
\end{prop}
\begin{proof}
    Again w.l.o.g.\ we can assume that $q<\infty$ and $\lambda=2^j$ with $j\in\N$; see \autoref{rem:reduction}.
    For some $\varepsilon<\min\{a-1,2-b\}$, let us choose $\eta \in \mathcal{S}(\R)\setminus\{0\}$ with $\supp \eta \subseteq B(0, \varepsilon)$ and let $\eta_m:= c_m \, \eta (\,\cdot - 2^{m} e_1)$, $m=1,\ldots,j$, where $c:=(c_m)_{m=1}^j \in\C^j \setminus\{0\}$ and $e_1$ denotes the first unit vector. 
    Thus, we have $\supp \eta_m \subseteq B(2^m e_1, \varepsilon)$ for all $m$. Based on these functions we define $f,\zeta\in \mathcal{S}(\R)$ by
    $$
        f := \sum_{m=1}^j \mathcal{F}^{-1} \eta_m \qquad\text{and}\qquad \zeta:=D_{2^{-j}}f.
    $$
    Then $\mathcal{F} (D_{2^j} \zeta) = \mathcal{F}f = \sum_{m=1}^j \eta_m$. 
    Since for $m=1,\ldots,j$ and $x\in B(2^m e_1, \varepsilon)$ there holds
    \begin{align}\label{eq:proof_varphi_k}
        \varphi_k(x) 
        =\begin{cases}
            1, & k=m,\\
            0, & k\in\N_0\setminus\{m\},
        \end{cases}
    \end{align}
    we see that $\varphi_k \, \mathcal{F} (D_{2^j} \zeta) = \eta_k$ if $k\in\{1,\ldots,j\}$ and $\varphi_k \, \mathcal{F} (D_{2^j} \zeta)=0$ if $k=0$ or $k>j$.
    Therefore,
    \begin{align*}
        \norm{D_{2^j} \zeta \sep \mathcal{E}^{s}_{u,p,q}(\R)} 
        &=\norm{ \bigg( \sum_{k=0}^\infty 2^{k s q} \abs{\big[\mathcal{F}^{-1} \big(\varphi_k \, \mathcal{F}[D_{2^j}\zeta] \big)\big](\cdot)}^q  \bigg)^{1/q}\sep \mathcal{M}^{u}_p(\R)} \\
        &=\norm{ \bigg( \sum_{m=1}^j 2^{m s q} \abs{(\mathcal{F}^{-1}\eta_m)(\cdot)}^q  \bigg)^{1/q}\sep \mathcal{M}^{u}_p(\R)}, 
    \end{align*}
    where $(\mathcal{F}^{-1} \eta_m)(x) = c_m\, (\mathcal{F}^{-1} \eta)(x) \exp\big(\mathrm{i}\, 2^{m} e_1 x \big)$ on $\R$ for all $m=1,\ldots,j$, such that 
    \begin{align}
        \norm{D_{2^j} \zeta \sep \mathcal{E}^{s}_{u,p,q}(\R)} 
        &=\norm{ \bigg( \sum_{m=1}^j 2^{m s q} \, \abs{c_m}^q \abs{ \big( \mathcal{F}^{-1} \eta\big)(\cdot)}^q  \bigg)^{1/q}\sep \mathcal{M}^{u}_p(\R)} \nonumber\\
        &=\bigg( \sum_{m=1}^j \big( 2^{m s} \, \abs{c_m} \big)^q \bigg)^{1/q} \norm{ \mathcal{F}^{-1} \eta \sep \mathcal{M}^{u}_p(\R)} \nonumber\\
        &\sim \bigg( \sum_{m=1}^j \big( 2^{m s} \, \abs{c_m} \big)^q \bigg)^{1/q}. \label{eq:proof_Tjzeta}
    \end{align}
    On the other hand, Fomula~\eqref{eq:FTj_phi} shows that
    $$
        \mathcal{F}\zeta 
        = \mathcal{F} (D_{2^{-j}}f) 
        = 2^{jd} \, D_{2^j}(\mathcal{F}f) 
        = 2^{jd} \, D_{2^j} \bigg(  \sum_{m=1}^j \eta_m \bigg) 
        = 2^{jd} \sum_{m=1}^j D_{2^j}\eta_m.
    $$
    Since $\varphi_0 = D_{2^j}( D_{2^{-j}}\varphi_0)$ and $D_{2^{-j}}\varphi_0\equiv 1$ on $B(0,2^j a) \supset \bigcup_{m=1}^j B(2^m e_1,\varepsilon)$, we thus obtain
    $$
        \varphi_0\, \mathcal{F}\zeta 
        = 2^{jd} \sum_{m=1}^j \big[D_{2^j}( D_{2^{-j}}\varphi_0)\big] D_{2^j}\eta_m 
        = 2^{jd} \sum_{m=1}^j D_{2^j}\big[ ( D_{2^{-j}}\varphi_0) \, \eta_m\big]
        = 2^{jd} \sum_{m=1}^j D_{2^j} \eta_m,
    $$
    such that (using \eqref{eq:FTj_phi} again) there holds
    $$
        \mathcal{F}^{-1} [ \varphi_0\, \mathcal{F}\zeta ] 
        = 2^{jd} \sum_{m=1}^j \mathcal{F}^{-1}[ D_{2^j} \eta_m]
        = \sum_{m=1}^j D_{2^{-j}}[\mathcal{F}^{-1} \eta_m]
        = D_{2^{-j}} f. 
    $$
    Moreover, for $k\in\N$ we have $D_{2^{-j}}\varphi_k = D_{2^{-j}} D_{2^{-k}}\psi = D_{2^{-(j+k)}}\psi=\varphi_{j+k}$ such that \eqref{eq:proof_varphi_k} implies $(D_{2^{-j}}\varphi_k)\,\eta_m=0$ for $m=1,\ldots,j$ and hence $\varphi_k\, \mathcal{F}\zeta = 2^{jd} \sum_{m=1}^j D_{2^j}\big[ ( D_{2^{-j}}\varphi_k) \, \eta_m\big]$ equals zero.
    Together this shows
     \begin{align*}
        \norm{ \zeta \sep \mathcal{E}^s_{u,p,q}(\R)}
        &= \norm{ \bigg( \sum_{k=0}^\infty 2^{ksq} \abs{\big(\mathcal{F}^{-1}[\varphi_k \, \mathcal{F}\zeta]\big)(\cdot)}^q \bigg)^{\frac{1}{q}} \sep \mathcal{M}^u_p(\R)} \\
        &= \norm{ \mathcal{F}^{-1}[\varphi_0 \, \mathcal{F}\zeta] \sep \mathcal{M}^u_p(\R)} \\
        &= \norm{ D_{2^{-j}}f \sep \mathcal{M}^u_p(\R)} \\
        &\sim 2^{j \frac{d}{u}} \norm{ f \sep \mathcal{M}^u_p(\R)},
    \end{align*}
    where we used \autoref{lem:M}(iii) for the last step.
    To analyze this further, we distinguish two cases. 
    If $u<1$, then we use the definition of $f$ and \autoref{lem:M}(i) to estimate
    \begin{align*}
        \norm{ f \sep \mathcal{M}^u_p(\R)}
        &\leq \norm{ \sum_{m=1}^j \abs{(\mathcal{F}^{-1}\eta_m)(\cdot)} \sep \mathcal{M}^u_p(\R)} 
        = \sum_{m=1}^j \abs{c_m} \norm{ \mathcal{F}^{-1} \eta \sep \mathcal{M}^{u}_p(\R)}
        \sim \sum_{m=1}^j \abs{c_m}.
    \end{align*}
 In case $u\geq 1$, we recall that $f=D_{2^j}\zeta$. Hence, setting $\overline{p}:=\max\{1,p\}$ we can employ \autoref{prop:E_in_M} in combination with \eqref{eq:proof_Tjzeta} to derive
    \begin{align*}
        \norm{ f \sep \mathcal{M}^u_p(\R)}
        = \norm{ D_{2^j}\zeta \sep \mathcal{M}^u_p(\R)} 
        \lesssim \norm{ D_{2^j}\zeta \sep \mathcal{E}^0_{u,\overline{p},2}(\R)} 
        \sim \bigg( \sum_{m=1}^j \abs{c_m}^2 \bigg)^{\frac{1}{2}}.
    \end{align*}
    In conclusion,
    \begin{align*}
        \norm{D_{2^j} \sep \mathcal{L}\big( \mathcal{E}^s_{u,p,q}(\R) \big)} 
        &\geq \frac{\norm{ D_{2^j}\zeta \sep \mathcal{E}^s_{u,p,q}(\R)}}{\norm{ \zeta \sep \mathcal{E}^s_{u,p,q}(\R)}} \\
        &\gtrsim 2^{-j\frac{d}{u}} \bigg( \sum_{m=1}^j \big( 2^{m s} \, \abs{c_m} \big)^q \bigg)^{1/q} \cdot \begin{cases}
            \displaystyle \bigg( \sum_{m=1}^j \abs{c_m} \bigg)^{-1}, & 0<u< 1,\\
            \displaystyle \bigg( \sum_{m=1}^j \abs{c_m}^2 \bigg)^{-\frac{1}{2}}, & 1 \leq u < \infty.
        \end{cases}
    \end{align*}
    Setting $c=(c_m)_{m=1}^j := e_j$ we obtain that $\norm{D_{2^j} \sep \mathcal{L}\big( \mathcal{E}^s_{u,p,q}(\R) \big)} \gtrsim 2^{j(s-\frac{d}{u})}$, while $c:=e_1$ gives the bound $\norm{D_{2^j} \sep \mathcal{L}\big( \mathcal{E}^s_{u,p,q}(\R) \big)} \gtrsim 2^{j(0-\frac{d}{u})}$. 
    For $s=0$ these lower bounds coincide. However, in this particular case we can also set $c:=(1,\ldots,1)$ to derive
    $$
        \norm{D_{2^j} \sep \mathcal{L}\big( \mathcal{E}^0_{u,p,q}(\R) \big)} 
        \gtrsim 2^{j(0-\frac{d}{u})} \cdot \begin{cases}
            \displaystyle j^{\frac{1}{q}-1}, & 0<u < 1,\\
            \displaystyle j^{\frac{1}{q}-\frac{1}{2}}, & 1 \leq u < \infty.
        \end{cases}
    $$
    Here for the special case $q = \infty$ we use $\frac{1}{\infty}= 0$. The proof is complete.
\end{proof}

If $0<p\leq 1$ and $s\leq \sigma_p$, we can partly improve our lower bounds of \autoref{prop:lower_bounds} using the ideas of Step 2 of the proof of~\cite[Thm.~2.2]{SchVy}, the local mean characterization of $\mathcal{E}^s_{u,p,q}(\R)$ given in \autoref{prop:local_mean} and the Gagliardo-Nierenberg inequality proven in \autoref{lem:gagliardo}. 
\begin{prop}\label{prop:local_mean_lower_bound}
    Let $0<p\leq u < \infty$, $0<q\leq\infty$ as well as $s\in\mathbb{R}$ and $\lambda\geq 2$.
    \begin{enumerate}
        \item If $0<p\leq 1$ and $s = \sigma_p$, then there holds
        \begin{align*}
            \norm{D_{\lambda} \sep \mathcal{L}\big(\mathcal{E}^{\sigma_p}_{u,p,q}(\R)\big)}
            &\gtrsim \lambda^{\sigma_p-\frac{d}{u}} \, \lambda^{d(\frac{1}{u}-\frac{1}{p})} \, (\log_2 \lambda)^{\frac{1}{p}}.
        \end{align*}

        \item If $0<p < 1$ and $s<\sigma_p$, then
        \begin{align*}
            \norm{D_{\lambda} \sep \mathcal{L}\big(\mathcal{E}^{s}_{u,p,q}(\R)\big)}
            &\gtrsim \lambda^{d(\frac{1}{u}-1)-\frac{d}{u}}.
        \end{align*}
    \end{enumerate}
\end{prop}

\begin{remark}\label{rem:lower}
    Recall that in \autoref{prop:lower_bounds} we already showed the lower bound $\lambda^{\max\{s,0\}-\frac{d}{u}}$ for all parameter constellations. Thus, for $0<p<1$ and $s<\sigma_p$ it follows
    \begin{align*}
            \norm{D_{\lambda} \sep \mathcal{L}\big(\mathcal{E}^{s}_{u,p,q}(\R)\big)}
            &\gtrsim \lambda^{\max\{s,0,d[\frac{1}{u}-1]\}-\frac{d}{u}}
            = \lambda^{\max\{s,\sigma_u\}-\frac{d}{u}}, \qquad \lambda\geq 2.
        \end{align*}
    Further, \autoref{prop:local_mean_lower_bound}(i) obviously yields an improvement only if $u=p$.
\end{remark}

\begin{proof}[Proof (of {\autoref{prop:local_mean_lower_bound}})]
    \emph{Step 1 (Preparations). }
    Once more, it suffices to prove the claim for $\lambda=2^j$ with $j\in\N$; see \autoref{rem:reduction}. 
    To this end, let us fix $\zeta\in\mathcal{D}(\R)$ with
    $$
        \zeta \geq 0, \qquad
        \supp \zeta \subseteq S:=B\Big(0, \frac{1}{8}\Big) \subseteq \R
        \qquad \text{and} \qquad \int_{\R} \zeta(x) \d x=1.
    $$
    Then for all parameters there holds $\zeta, D_{2^j}\zeta \in \mathcal{S}(\R)\subseteq \mathcal{E}^s_{u,p,q}(\R)$ and 
    \begin{align}\label{eq:proof_lower_zeta}
        \norm{D_{2^j} \sep \mathcal{L}\big(\mathcal{E}^s_{u,p,q}(\R)\big)}
        &\geq \frac{\norm{D_{2^j}\zeta \sep \mathcal{E}^s_{u,p,q}(\R)}}{\norm{\zeta \sep \mathcal{E}^s_{u,p,q}(\R)}} 
        \sim \norm{D_{2^j}\zeta \sep \mathcal{E}^s_{u,p,q}(\R)}.
    \end{align}

    \emph{Step 2 (Core Estimate). }
    Here we show that
    \begin{align}\label{eq:proof_core_estimate}
        \norm{D_{2^j}\zeta \sep \mathcal{E}^s_{u,p,q}(\R)}
        &\gtrsim 2^{j(s-\frac{d}{p})} \cdot 
        \begin{cases}
            \displaystyle \bigg( \sum_{\ell=0}^{j-1} 2^{\ell(d\,[\frac{1}{p}-1]-s)p}  \bigg)^{\frac{1}{p}}, & 0<q<\infty,\\
            \sup_{\ell=0,\ldots,j-1}\limits 2^{\ell(d\,[\frac{1}{p}-1]-s)}, & q=\infty.
        \end{cases}
    \end{align} 
    To this end, as stated above, we like to employ \autoref{prop:local_mean}. Therein, $\eta\in\mathcal{S}(\R)$ may be chosen in a way such that $K:=\Delta^{2N}\eta$ satisfies
$$
    K(y) \geq c > 0 \qquad\text{on}\qquad \Omega := \left\{ y\in\R : 1+\frac{1}{4}  \leq \abs{y} \leq 2-\frac{1}{4} \right\}.
$$
Let us briefly sketch how to construct such an $\eta$. 
Recall that the Laplacian of a rotationally symmetric function $x\mapsto \varrho(x):=\omega(\abs{x})$ on $\R$ with $\omega\in C^\infty([0,\infty))$ is given by 
$$
    (\Delta\varrho)(x) = \Big(\Big[\partial_r^2 + \frac{d-1}{r}\,\partial_r\Big]\,\omega\Big)(r), \qquad r:=\abs{x}>0.
$$
So, the condition $(\Delta^{2N}\varrho)(x)\equiv 1$ for all $x\in\R$ with $\abs{x} \in \mathcal{I}:=(1,R)\subseteq (1,\infty)$ can be considered as a $4N$-th order ODE for $\omega\colon \mathcal{I} \to \mathbb{R}$, $r\mapsto \omega(r)$, with analytic coefficients and smooth right-hand side which might be rewritten as a first order linear system. Prescribing positive initial values at $r_0:=1$, we find a smooth solution $\widetilde{\omega}$ on the interval $\mathcal{I}$ which stays strictly positive on $\mathcal{I}$, provided that $R>1$ is chosen sufficiently small. 
Appropriate rescaling, dilation and shifting $\omega(r) := c_1\, \widetilde{\omega}(c_2\, r + c_3)$ leads to a smooth function $\varrho :=\omega(\abs{\cdot})$ on the annulus $\Omega':= \{y\in \R : 1+\frac{1}{8}\leq \abs{y} < 2-\frac{1}{8}\}$ with
$$
    (\Delta^{2N}\varrho)(x) \equiv c>0
    \qquad\text{and}\qquad 0\leq \varrho(x) \leq 1, \qquad x\in \Omega'.
$$
Next we extend $\omega$ to the full halfline $[0,\infty)$ by setting $\omega(r):\equiv 1$ for $0 \leq r\leq 1+\frac{1}{16}$ as well as $\omega(r):\equiv 0$ for $r\geq 2-\frac{1}{16}$ and linear interpolation on the remaining parts. Then $\varrho=\omega(\abs{\cdot})$ satisfies $0\leq \varrho(x) \leq 1$ on the whole of $\R$ while still $(\Delta^{2N}\varrho)(x) \equiv c$ for $x\in\Omega'$. 
However, it is no longer smooth. 
Therefore, in a last step we may take the convolution of~$\varrho$ with some smooth mollifier in order to obtain $\eta\in\mathcal{S}(\R)$ with the desired properties. 

If we now assume that $q<\infty$, then from \autoref{prop:local_mean} and \autoref{lem:M}(iii) it follows
\begin{align}
    \norm{D_{2^j}\zeta \sep \mathcal{E}^s_{u,p,q}(\R)}
    &\geq \norm{ \bigg( \sum_{k=1}^j 2^{ksq} \abs{(K_k \ast [D_{2^j}\zeta])(\cdot)}^q \bigg)^{\frac{1}{q}} \sep \mathcal{M}^u_p(\R)} \nonumber\\
    &= 2^{js} \, 2^{-js} \norm{ D_{2^j}\bigg( \sum_{k=1}^j 2^{ksq} \abs{ \big[ D_{2^{-j}}\big(K_k \ast [D_{2^j}\zeta]\big)\big] (\cdot)}^q \bigg)^{\frac{1}{q}} \sep \mathcal{M}^u_p(\R)} \nonumber\\
    &\sim 2^{j(s-\frac{d}{u})} \norm{ \bigg( \sum_{k=1}^j 2^{(k-j)sq} \abs{ \big[ D_{2^{-j}}\big(K_k \ast [D_{2^j}\zeta]\big)\big] (\cdot)}^q \bigg)^{\frac{1}{q}} \sep \mathcal{M}^u_p(\R)}. \label{eq:proof_D2jzeta}
\end{align}
Note that $K_k, D_{2^j}\zeta\in \mathcal{D}(\R)$ such that the latter convolution can be written as
\begin{align*}
    (K_k \ast [D_{2^j}\zeta])(x)
    &= \int_{\R} K_k(x-y) \, [D_{2^j}\zeta](y) \d y \\
    &= \int_{\R} 2^{kd} \, [D_{2^k}K](x-y) \, [D_{2^j}\zeta](y) \d y \\
    &= 2^{kd} \int_{\R} K(2^k[x-y]) \, \zeta(2^j y) \d y \\
    &= 2^{(k-j)d} \int_{\R} K(2^k[x-2^{-j}z]) \, \zeta(z) \d z,
    \qquad x\in\R,
\end{align*}
since $z:=2^j y$ gives $y=2^{-j}z$ and $\!\d y = 2^{-jd} \d z$.
Therefore, as $\supp\zeta \subseteq S$,
\begin{align*}
    D_{2^{-j}}(K_k \ast [D_{2^j}\zeta])(x)
    &= (K_k \ast [D_{2^j}\zeta])(2^{-j}x) \\
    &= 2^{(k-j)d} \int_{\R} K(2^k[2^{-j}x-2^{-j}z]) \, \zeta(z) \d z \\
    &= 2^{(k-j)d} \int_{S} K(2^{k-j}[x-z]) \, \zeta(z) \d z \\
    &=: 2^{(k-j)d} I_{k-j}(x)
\end{align*}
and hence (setting $\ell:=j-k$)
\begin{align*}
    F(x)
    &:=\bigg( \sum_{k=1}^j 2^{(k-j)sq} \abs{D_{2^{-j}}(K_k \ast [D_{2^j}\zeta])(x)}^q \bigg)^{\frac{1}{q}} \\
    &= \bigg( \sum_{\ell=0}^{j-1} 2^{-\ell(s+d)q} \abs{I_{-\ell}(x)}^q \bigg)^{\frac{1}{q}} \\ 
    &\geq \bigg( \sum_{\ell=0}^{j-1} 2^{-\ell(s+d)q} \, \chi_{B_\ell}(x) \abs{I_{-\ell}(x)}^q \bigg)^{\frac{1}{q}},
    \qquad x\in\R,
\end{align*}
where we set $B_\ell := B(x_\ell,R_\ell)$ with $x_\ell := \frac{12}{8}\, 2^\ell e_1$ and $R_\ell:=\frac{1}{8}\, 2^\ell$. 
Note that for all $x\in B_\ell$ and $z\in S$ there holds $2^{-\ell}[x-z]\in\Omega$ such that
$$
    I_{-\ell}(x) 
    = \int_{S} K(2^{-\ell}[x-z]) \, \zeta(z) \d z
    \geq c \int_{S} \zeta(z) \d z
    = c \int_{\R} \zeta(z) \d z 
    \gtrsim 1, \qquad x\in B_\ell.
$$
Moreover, all balls $B_\ell$ are mutually disjoint and contained in $B^j :=B(0,2^j)$.
Therefore, we conclude
\begin{align*}
    F(x) 
    &\gtrsim \bigg( \sum_{\ell=0}^{j-1} 2^{-\ell(s+d)q} \, \chi_{B_\ell}(x) \bigg)^{\frac{1}{q}} 
    = \sum_{\ell=0}^{j-1} 2^{-\ell(s+d)} \, \chi_{B_\ell}(x) , 
    \qquad x\in\R,
\end{align*}
such that for all balls $B\subseteq \R$ there holds
\begin{align*}
    \int_{B} \abs{F(x)}^p \d x 
    &= \int_{\R} \abs{\chi_B(x) \, F(x)}^p \d x\\
    &\gtrsim \int_{\R} \abs{\sum_{\ell=0}^{j-1} 2^{-\ell(s+d)} \, \chi_{B\cap B_\ell}(x)}^p \d x  \\
    &= \sum_{\ell=0}^{j-1} 2^{-\ell(s+d)p} \int_{\R} \chi_{B\cap B_\ell}(x)\d x  \\
    &= \sum_{\ell=0}^{j-1} 2^{-\ell(s+d)p} \abs{ B\cap B_\ell}. 
\end{align*}
Choosing $B:=B^j$ as above, this allows to bound the Morrey (quasi-)norm of $F$ by
\begin{align}
    \norm{F \sep \mathcal{M}^u_p(\R)}
    &\gtrsim \sup_{B} \abs{B}^{\frac{1}{u}-\frac{1}{p}} \bigg( \sum_{\ell=0}^{j-1} 2^{-\ell(s+d)p} \abs{ B\cap B_\ell}  \bigg)^{\frac{1}{p}} \nonumber\\
    &\gtrsim 2^{jd(\frac{1}{u}-\frac{1}{p})} \bigg( \sum_{\ell=0}^{j-1} 2^{-\ell(s+d)p} \, 2^{\ell d}  \bigg)^{\frac{1}{p}} \nonumber\\
    &= 2^{jd(\frac{1}{u}-\frac{1}{p})} \bigg( \sum_{\ell=0}^{j-1} 2^{\ell(d\,[\frac{1}{p}-1]-s)p}  \bigg)^{\frac{1}{p}}. \label{eq:proof_F_new}
\end{align}
Together with \eqref{eq:proof_D2jzeta} this proves \eqref{eq:proof_core_estimate}, provided that $q<\infty$. 

If $q=\infty$, we can argue as before to see that 
\begin{align*}
    \norm{D_{2^j} \sep \mathcal{L}\big(\mathcal{E}^s_{u,p,q}(\R)\big)}
    &\gtrsim 2^{j(s-\frac{d}{u})} \norm{ \widetilde{F} \sep \mathcal{M}^u_p(\R)},
\end{align*}
where now $\widetilde{F}:=\max_{k=1,\ldots,j}\limits 2^{(k-j)s} \abs{D_{2^{-j}}(K_k \ast [D_{2^j}\zeta])(\cdot)}$ such that \eqref{eq:proof_F_new} has to be replaced by
\begin{align*}
    \norm{\widetilde{F} \sep \mathcal{M}^u_p(\R)}
    &\gtrsim 2^{jd(\frac{1}{u}-\frac{1}{p})} \max_{\ell=0,\ldots, j-1} 2^{\ell(d\,[\frac{1}{p}-1]-s)}.
\end{align*}

\emph{Step 3 (Conclusion). }
    Now assertion~(ii) for $\lambda=2^j$ with $j\in\N$ easily follows from \eqref{eq:proof_lower_zeta} and \eqref{eq:proof_core_estimate}.
    Indeed, for $p < 1$ and $s<\sigma_p=d(\frac{1}{p}-1)$ we can conclude
    \begin{align*}
        \norm{D_{2^j} \sep \mathcal{L}\big(\mathcal{E}^s_{u,p,q}(\R)\big)}
        &\gtrsim 2^{j(\sigma_p-\frac{d}{p})} 
        = 2^{-jd} 
        = 2^{j [d(\frac{1}{u}-1) - \frac{d}{u}]}.
    \end{align*}
    For $p\leq 1$ and $s=\sigma_p$, Formula~\eqref{eq:proof_core_estimate} yields
    \begin{align}
    \label{eq:proof_sigma_p_zeta}
        \norm{D_{2^j}\zeta \sep \mathcal{E}^{\sigma_p}_{u,p,q}(\R)}
        &\gtrsim 2^{j(\sigma_p-\frac{d}{p})} \, j^{\frac{1}{p}} 
        = 2^{-jd} \, j^{\frac{1}{p}}
        \qquad\text{if}\quad q<\infty.
    \end{align}
    So, in order to show (i), it thus suffices to prove that \eqref{eq:proof_sigma_p_zeta} remains true for $q=\infty$.
    For this purpose, we use the Gagliardo-Nierenberg inequality (\autoref{lem:gagliardo}) for $f:=D_{2^j}\zeta$, where we let $s_0:=\sigma_{p}$ as well as $p_0:=p \leq 1$, $u_0:=u$ and $q_0:=\infty$. 
    Further, we choose $0<q_1:=u_1:=p_1 < p_0$ and set $s_1:=\sigma_{p_1}$ as well as $r:=1$.
    Then $p_1<p_\Theta<p_0 \leq 1$ and $s_\Theta=\sigma_{p_\Theta}$.
    On the one hand, \autoref{thm:upper_bounds} yields
\begin{align*}
    \norm{f \sep \mathcal{E}^{s_1}_{u_1,p_1,q_1}(\R)}
    &\leq \norm{D_{2^j} \sep \mathcal{L}\big( \mathcal{E}^{\sigma_{p_1}}_{p_1,p_1,p_1}(\R) \big)} \norm{\zeta \sep \mathcal{E}^{\sigma_{p_1}}_{p_1,p_1,p_1}(\R)}
    \lesssim 2^{j(\sigma_{p_1}-\frac{d}{p_1})} \, j^{\frac{1}{p_1}}
    = 2^{-jd} \, j^{\frac{1}{p_1}},
\end{align*}
i.e., $\norm{f \sep \mathcal{E}^{s_1}_{u_1,p_1,q_1}(\R)}^\Theta
    \lesssim 2^{-jd\Theta} \,  j^{\frac{\Theta}{p_1}}$, and on the other hand, using \eqref{eq:proof_sigma_p_zeta} we find
\begin{align*}
    \norm{f \sep \mathcal{E}^{s_\Theta}_{u_\Theta,p_\Theta,r}(\R)}
    = \norm{D_{2^j}\zeta \sep \mathcal{E}^{\sigma_{p_\Theta}}_{u_\Theta,p_\Theta,1}(\R)}
    \gtrsim 2^{-jd} \, j^{\frac{1}{p_\Theta}}
    = 2^{-jd\Theta} \, j^{\frac{\Theta}{p_1}} \, 2^{-jd(1-\Theta)} \, j^{\frac{1-\Theta}{p_0}}.
\end{align*}
Therefore, \autoref{lem:gagliardo} implies
\begin{align*}
    2^{-jd\Theta} \, j^{\frac{\Theta}{p_1}} \, 2^{-jd(1-\Theta)} \, j^{\frac{1-\Theta}{p_0}}
    \lesssim 
    \norm{D_{2^j}\zeta \sep \mathcal{E}^{s_0}_{u_0,p_0,q_0}(\R)}^{1-\Theta} 2^{-jd\Theta} \, j^{\frac{\Theta}{p_1}}
\end{align*}
and hence
\begin{align*}
    \norm{D_{2^j}\zeta \sep \mathcal{E}^{\sigma_{p}}_{u,p,\infty}(\R)}
    &\gtrsim 2^{-jd} \, j^{\frac{1}{p}}
\end{align*}
which completes the proof.
\end{proof}

\begin{remark}\label{rem:correctionFspq}
    Let us record that in Step 2 of the proof of~\cite[Thm.~2.2]{SchVy} we also have to distinguish whether $q$ is finite or not. 
    Consequently, contrary to the claim in \cite[Thm.~2.2]{SchVy}, the by now best bounds for the (quasi-)norm of dilation operators on the Triebel-Lizorkin scale $F^s_{p,q}(\R)$ with $s=\sigma_p$, $0<p\leq 1$ and $q=\infty$ read
    $$
        2^{j(\sigma_p-\frac{d}{p})} \lesssim \norm{D_{2^j} \sep \mathcal{L}\big(F^{\sigma_p}_{p,\infty}(\R) \big)} \lesssim 2^{j(\sigma_p-\frac{d}{p})} \, j^{\frac{1}{p}}, \qquad j\in\N.
    $$
    Thus, our result yields an improvement here.
\end{remark}

\subsection{Proofs of \autoref{thm:main} and \autoref{cor:main}}\label{sec_put_together_1}
The case of small $\lambda$ in \autoref{thm:main} is covered by \autoref{prop:T_mu} while for $\lambda\geq 2$ we simply have to combine the upper bounds from \autoref{thm:upper_bounds} with the lower bounds in Propositions~\ref{prop:lower_bounds} and \ref{prop:local_mean_lower_bound} resp.\ \autoref{rem:lower}. Note that in Part (iii) of Theorem \ref{thm:main} the conditions $s=0$ and $1\leq u$ imply that $\max\{s,\sigma_u\}=0$. Moreover, \autoref{cor:main} is a direct consequence of \autoref{thm:main} and \autoref{l_bp1}(iv).

\section{Some New Equivalent (Quasi-)Norms for \texorpdfstring{$\mathcal{E}^s_{u,p,q}(\R)$}{EsupqRd}}\label{sect:characterizations}
As an application, in this final \autoref{sect:characterizations} we derive some new equivalent (quasi-)norms for Triebel-Lizorkin-Morrey spaces $\mathcal{E}^s_{u,p,q}(\mathbb{R}^{d})$. To start with, we show that for the full range of parameters the following quantities are not only equivalent to $\norm{\cdot \sep \mathcal{E}^s_{u,p,q}(\mathbb{R}^{d})}$, they actually also allow to characterize $\mathcal{E}^s_{u,p,q}(\mathbb{R}^{d})$; see \autoref{thm:characterization} below.
\begin{defi}
    Let $0 < p \leq u < \infty$, $0 < q \leq \infty$ and $s \in\mathbb{R}$. 
    Then for $f\in\mathcal{S}'(\R)$ we let
    $$
        \norm{f \sep \mathcal{E}^s_{u,p,q}(\R)}_K^{\clubsuit} 
        := \norm{ \mathcal{F}^{-1} [ \varphi_0 \, \mathcal{F}f ] \sep \mathcal{M}^{u}_{p}(\R)} + \abs{f}_K, \qquad K\in\N_0\cup\{\infty\},
    $$
    and
    $$
        \norm{f \sep \mathcal{E}^s_{u,p,q}(\R)}_K
        := 2^{-Ks} \norm{ \mathcal{F}^{-1} \big[ (D_{2^{K+1}}\varphi_0) \, \mathcal{F}f \big] \sep \mathcal{M}^{u}_{p}(\R)} + \abs{f}_K,
        \qquad K\in\N_0,
    $$
    where (with suitable modification for $q=\infty$) we set
    $$
        \abs{f}_K :=  \norm{ \bigg( \sum_{k=-K}^\infty 2^{ksq} \abs{\mathcal{F}^{-1}[\psi_k\,\mathcal{F}f](\cdot)}^q \bigg)^{\frac{1}{q}} \sep \mathcal{M}^u_p(\R)},
        \qquad K\in\mathbb{Z}\cup\{\infty\}.
    $$
\end{defi}

Our proof is based on the new Fourier multiplier \autoref{prop:FM_M} which allows to derive the subsequent estimates.
\begin{lemma}\label{lem:NormK}
    Let $f\in\mathcal{S}'(\R)$ as well as $0 < p \leq u < \infty$, $0 < q \leq \infty$ and $s\in\mathbb{R}$.
    \begin{enumerate}
        \item Set $\mu:=\min\{p,q\}$. Then $\abs{f}_{-1} 
                \lesssim \norm{f \sep \mathcal{E}^s_{u,p,q}(\R)}$
        while for $K\in\N_0\cup\{\infty\}$ we have
            \begin{align*}
                \abs{f}_K 
                \lesssim \bigg( \sum_{\ell=0}^K 2^{\ell(\sigma_p-s) \mu} \bigg)^{\frac{1}{\mu}} \norm{f \sep \mathcal{E}^s_{u,p,q}(\R)}.
            \end{align*}
        \item For $K\in\N_0$ there holds
            \begin{align*}
                \norm{\mathcal{F}^{-1}\big[ (D_{2^{K+1}}\varphi_0) \mathcal{F}f \big] \sep \mathcal{M}^u_p(\R)}
                \lesssim 2^{K\sigma_p} \norm{\mathcal{F}^{-1} [ \varphi_0 \, \mathcal{F}f ] \sep \mathcal{M}^u_p(\R)}
            \end{align*}
            and
            \begin{align*}
                \norm{ \mathcal{F}^{-1} [ \varphi_0 \, \mathcal{F}f ] \sep \mathcal{M}^{u}_{p}(\R)}
                &\lesssim \norm{ \mathcal{F}^{-1} [ (D_{2^{K+1}}\varphi_0) \, \mathcal{F}f ] \sep \mathcal{M}^{u}_{p}(\R)} + \bigg( \sum_{\ell=0}^{K} 2^{\ell s} \bigg) \abs{f}_{K}.
            \end{align*}
    \end{enumerate}
    Therein all implied constants are independent of $f$ and $K$.
\end{lemma}
\begin{proof}
    W.l.o.g.\ assume $q<\infty$ and let $u_k := \mathcal{F}^{-1} [ \psi_k \, \mathcal{F}f ] = \mathcal{F}^{-1} \big[ (D_{2^{-k}} \psi) \, \mathcal{F}f \big]$ for $k\in\mathbb{Z}$.

    \emph{Step 1. } 
    To show (i) we may assume that $f\in\mathcal{E}^s_{u,p,q}(\mathbb{R}^{d})$. Since $\psi_k=\varphi_k$ for $k\in\N$, we have
    \begin{align*}
        \norm{f \sep \mathcal{E}^s_{u,p,q}(\R)} 
        = \norm{\bigg( \abs{\mathcal{F}^{-1} [ \varphi_0 \, \mathcal{F}f ](\cdot)}^q + \sum_{k=1}^\infty 2^{ksq} \abs{u_k(\cdot)}^q \bigg)^{\frac{1}{q}} \sep \mathcal{M}^{u}_{p}(\R)}
    \end{align*}
    such that using \autoref{lem:M}(i) and (ii) we conclude
    \begin{align}\label{eq:proof_K}
        \norm{f \sep \mathcal{E}^s_{u,p,q}(\R)} 
        \sim \norm{ \mathcal{F}^{-1} [ \varphi_0 \, \mathcal{F}f ] \sep \mathcal{M}^{u}_{p}(\R)} + \abs{f}_{-1}.
    \end{align}
    This shows (i) for $K=-1$. 
    If $K\in\N_0$, then
    \begin{align*}
        \bigg( \sum_{k=-K}^{\infty} 2^{ksq} \abs{u_k(x)}^q \bigg)^{\frac{1}{q}}
        \lesssim \bigg( \sum_{k=-K}^{0} 2^{ksq}\abs{u_k(x)}^q \bigg)^{\frac{1}{q}} + \bigg( \sum_{k=1}^{\infty} 2^{ksq}\abs{ u_k(x)}^q \bigg)^{\frac{1}{q}}, \qquad x\in\R,
    \end{align*}
    such that from \autoref{lem:M}(i) and $\varphi_k=\psi_k$ if $k\in\N$ it follows
    \begin{align*}
        \abs{f}_K
        &= \norm{ \bigg( \sum_{k=-K}^{\infty} 2^{k sq} \abs{u_{k}(\cdot)}^q \bigg)^{\frac{1}{q}} \sep \mathcal{M}^u_p(\R)} \\
        &\lesssim \norm{ \bigg( \sum_{\ell=0}^{K} 2^{-\ell sq} \abs{u_{-\ell}(\cdot)}^q \bigg)^{\frac{1}{q}} \sep \mathcal{M}^u_p(\R)} + \norm{f \sep \mathcal{E}^s_{u,p,q}(\R)}.
    \end{align*}
    Now \autoref{prop:FM_M} can be used to show that with $\mu=\min\{p,q\}$ there holds
    \begin{align*}
        \norm{ \bigg( \sum_{\ell=0}^{K} 2^{-\ell sq} \abs{ u_{-\ell}(\cdot)}^q \bigg)^{\frac{1}{q}} \sep \mathcal{M}^u_p(\R) }
        &\lesssim \bigg( \sum_{\ell=0}^{K} 2^{\ell(\sigma_p-s)\mu} \bigg)^{\frac{1}{\mu}} \norm{f \sep \mathcal{E}^{s}_{u,p,q}(\R)},
    \end{align*}
    see \eqref{eq:proof_II} and \eqref{eq:proof_A2} for details.
    This proves (i) for $K\in\N_0$.
    However, exactly the same reasoning can be used if $K=\infty$.

\emph{Step 2. } 
For the first bound in (ii) we may assume that $f_0:=\mathcal{F}^{-1} [ \varphi_0 \, \mathcal{F}f ] \in \mathcal{M}^u_p(\R)$. 
Then the desired estimate follows from \autoref{prop:FM_M}; see \eqref{eq:step1}.
To show also the second bound, we note that $\varphi_0 = D_{2^{K+1}}\varphi_0 + \psi_{-K} + \ldots + \psi_0$. Hence, with $u_k$ as above, \autoref{lem:M}(i) gives
\begin{align*}
    \norm{ f_0 \sep \mathcal{M}^{u}_{p}(\R)}
    &= \norm{ \mathcal{F}^{-1} [ (D_{2^{K+1}}\varphi_0) \, \mathcal{F}f ] + \sum_{k=-K}^0 u_k \sep \mathcal{M}^{u}_{p}(\R)} \\
    &\lesssim \norm{ \mathcal{F}^{-1} [ (D_{2^{K+1}}\varphi_0) \, \mathcal{F}f ] \sep \mathcal{M}^{u}_{p}(\R)} + \norm{\sum_{k=-K}^0 \abs{ u_k(\cdot)} \sep \mathcal{M}^{u}_{p}(\R)}.
\end{align*}
Moreover, for all $x\in\R$ we have
\begin{align*}
    \sum_{k=-K}^0 \abs{u_k(x)}
    \leq \max_{j=-K,\ldots,0} \abs{2^{js}\,u_j(x)} \sum_{k=-K}^0 2^{-ks} 
    \leq \bigg( \sum_{j=-K}^\infty \abs{2^{js}\,u_j(x)}^q \bigg)^{\frac{1}{q}} \sum_{\ell=0}^{K} 2^{\ell s}
\end{align*}
which shows $\norm{\sum_{k=-K}^0 \abs{ u_k(\cdot)} \sep \mathcal{M}^{u}_{p}(\R)} \leq \big( \sum_{\ell=0}^{K} 2^{\ell s} \big) \abs{f}_{K}$ and hence
\begin{align*}
    \norm{ \mathcal{F}^{-1} [ \varphi_0 \, \mathcal{F}f ] \sep \mathcal{M}^{u}_{p}(\R)}
    &\lesssim \norm{ \mathcal{F}^{-1} [ (D_{2^{K+1}}\varphi_0) \, \mathcal{F}f ] \sep \mathcal{M}^{u}_{p}(\R)} + \bigg( \sum_{\ell=0}^{K} 2^{\ell s} \bigg) \abs{f}_{K}
\end{align*}
as claimed.
\end{proof}

Now the announced characterization reads as follows.
\begin{theorem}\label{thm:characterization}
    Let $ 0 < p \leq u < \infty$, $0 < q \leq \infty$ as well as $s \in\mathbb{R}$ and $K\in\N_0$.
    Then the expressions $\norm{\,\cdot \sep \mathcal{E}^s_{u,p,q}(\R)}$, $\norm{\,\cdot \sep \mathcal{E}^s_{u,p,q}(\R)}_K^{\clubsuit}$ and $\norm{\,\cdot \sep \mathcal{E}^s_{u,p,q}(\R)}_K$ 
    are mutually equivalent on $\mathcal{S}'(\R)$.
    If $s>\sigma_p$, then also $\norm{\,\cdot \sep \mathcal{E}^s_{u,p,q}(\R)}_\infty^{\clubsuit} \sim \norm{\,\cdot \sep \mathcal{E}^s_{u,p,q}(\R)}$ on $\mathcal{S}'(\R)$.
\end{theorem}
\begin{proof}
Let $f\in\mathcal{S}'(\R)$ and w.l.o.g.\ assume $q<\infty$.

\emph{Step 1. }  We show that $\norm{f \sep \mathcal{E}^s_{u,p,q}(\R)}_K^{\clubsuit} \sim_K \norm{f \sep \mathcal{E}^s_{u,p,q}(\R)}$ for all $K\in\N_0$. 
    As clearly $\abs{f}_{-1} \leq \abs{f}_{K}$, Formula \eqref{eq:proof_K} proves 
    $$
        \norm{f \sep \mathcal{E}^s_{u,p,q}(\R)} 
        \sim \norm{ \mathcal{F}^{-1} [ \varphi_0 \, \mathcal{F}f ] \sep \mathcal{M}^{u}_{p}(\R)} + \abs{f}_{-1}
        \lesssim \norm{f \sep \mathcal{E}^s_{u,p,q}(\R)}_K^{\clubsuit}.
    $$
    On the other hand, with $\mu:=\min\{p,q\}$ Formula~\eqref{eq:proof_K} and \autoref{lem:NormK}(i) imply
    \begin{align*}
        \norm{f \sep \mathcal{E}^s_{u,p,q}(\R)}_K^{\clubsuit} 
        \lesssim \norm{f \sep \mathcal{E}^s_{u,p,q}(\R)} + \bigg( \sum_{\ell=0}^K 2^{\ell(\sigma_p-s) \mu} \bigg)^{\frac{1}{\mu}} \norm{f \sep \mathcal{E}^s_{u,p,q}(\R)}
        \lesssim_K \norm{f \sep \mathcal{E}^s_{u,p,q}(\R)}.
    \end{align*}

    \emph{Step 2. } To show that also $\norm{f \sep \mathcal{E}^s_{u,p,q}(\R)} \sim_K \norm{f \sep \mathcal{E}^s_{u,p,q}(\R)}_K$ we again let $K\in\N_0$. Then from \autoref{lem:NormK}(ii) as well as \eqref{eq:proof_K} and \autoref{lem:NormK}(i) it follows 
    \begin{align}
        \norm{f \sep \mathcal{E}^s_{u,p,q}(\R)}_K
        &\lesssim 2^{K(\sigma_p-s)} \norm{ \mathcal{F}^{-1} [ \varphi_0 \, \mathcal{F}f ] \sep \mathcal{M}^{u}_{p}(\R)} + \abs{f}_K \nonumber\\
        &\lesssim \bigg( \sum_{\ell=0}^K 2^{\ell(\sigma_p-s) \mu} \bigg)^{\frac{1}{\mu}} \norm{f \sep \mathcal{E}^s_{u,p,q}(\R)}.\label{eq:proof_upper_bound_K}
    \end{align}
    In turn, \eqref{eq:proof_K} and \autoref{lem:NormK}(ii) as well as $\abs{f}_{-1}\leq \abs{f}_K$ yield that
    \begin{align*}
        \norm{f \sep \mathcal{E}^s_{u,p,q}(\R)}
        &\lesssim \norm{ \mathcal{F}^{-1} [ (D_{2^{K+1}}\varphi_0) \, \mathcal{F}f ] \sep \mathcal{M}^{u}_{p}(\R)} + \bigg( \sum_{\ell=0}^{K} 2^{\ell s} \bigg) \abs{f}_{K} + \abs{f}_{-1} \\
        &\lesssim 2^{Ks} \, 2^{-Ks} \norm{ \mathcal{F}^{-1} [ (D_{2^{K+1}}\varphi_0) \, \mathcal{F}f ] \sep \mathcal{M}^{u}_{p}(\R)} + \bigg( \sum_{\ell=0}^{K} 2^{\ell s} \bigg) \abs{f}_{K} \\
        &\leq \bigg( \sum_{\ell=0}^{K} 2^{\ell s} \bigg) \norm{f \sep \mathcal{E}^s_{u,p,q}(\R)}_K.
    \end{align*}

    \emph{Step 3 (Case $s>\sigma_p$). } To complete the proof, we note that for $K=\infty$ and $s>\sigma_p$ the arguments of Step 1 remain valid such that also $\norm{f \sep \mathcal{E}^s_{u,p,q}(\R)}_\infty^{\clubsuit} \sim \norm{f \sep \mathcal{E}^s_{u,p,q}(\R)}$.
\end{proof}

\begin{remark}
    We note in passing that \autoref{thm:characterization} is in good agreement with the upper bounds we found in \autoref{thm:upper_bounds}. 
    Indeed, for $f \in \mathcal{E}^s_{u,p,q}(\R)$ a somewhat handwavy argument based on \eqref{eq:building_block_shift}, the fact that $D_{2^j}\psi_k=\psi_{k-j}$ for $k,j\in\N$ and \autoref{lem:M}(iii) indicates that $\abs{D_{2^j}f}_{-1} \sim 2^{j(s- \frac{d}{u})} \abs{f}_{j-1}$.
    Therefore, using \eqref{eq:proof_K}, \eqref{eq:building_block_shift}, \autoref{lem:M}(iii) and \eqref{eq:proof_upper_bound_K} we find 
    \begin{align*}
        \norm{D_{2^j}f \sep \mathcal{E}^s_{u,p,q}(\R)} 
        &\sim \norm{ \mathcal{F}^{-1} [ \varphi_0 \, \mathcal{F}(D_{2^j}f) ] \sep \mathcal{M}^{u}_{p}(\R)} + \abs{D_{2^j}f}_{-1} \\
        &\sim 2^{-j \frac{d}{u}} \norm{ \mathcal{F}^{-1} [ (D_{2^j}\varphi_0) \, \mathcal{F}f ] \sep \mathcal{M}^{u}_{p}(\R)} + 2^{j(s- \frac{d}{u})} \abs{f}_{j-1}  \\
        &\sim 2^{j(s-\frac{d}{u})} \norm{f \sep \mathcal{E}^s_{u,p,q}(\R)}_{j-1} \\
        &\lesssim 2^{j(s-\frac{d}{u})} \bigg( \sum_{\ell=0}^{j-1} 2^{\ell(\sigma_p-s) \mu} \bigg)^{\frac{1}{\mu}} \norm{f \sep \mathcal{E}^s_{u,p,q}(\R)}, \qquad j\in\N,
    \end{align*}
    where again $\mu=\min\{p,q\}$. This covers almost all statements of \autoref{thm:upper_bounds} if $\lambda=2^j$ with $j\in\N$ (which in view of \autoref{rem:reduction} is sufficient for results with general $\lambda>1$).
\end{remark}

Finally, let us take the opportunity to construct some further equivalent (quasi-)norms in $\mathcal{E}^s_{u,p,q}(\R)$ for parameters that ensure that $\mathcal{E}^s_{u,p,q}(\R) \hookrightarrow \mathcal{M}^u_p(\R)$, i.e.\ in a setting where we actually deal with \emph{functions} rather than distributions.
Our result particularly covers and extends a well-known assertion for Triebel-Lizorkin spaces $F^s_{p,q}(\R)=\mathcal{E}^s_{p,p,q}(\R)$ due to Triebel~\cite[Thm.~2.3.3(ii)]{Tr92}.

\begin{theorem}\label{thm:characterization_by_M}
    Let $0<p\leq u < \infty$, $0<q\leq \infty$ and $s>\frac{p}{u} \, \sigma_{p}$. 
    Further let $f\in \mathcal{E}^s_{u,p,q}(\R)$ and $K\in \{-1\} \cup \N_0$. 
    Then with constants independent of $f$ there holds
    $$
        \norm{f \sep \mathcal{E}^s_{u,p,q}(\R)}
        \sim_K \norm{f \sep \mathcal{E}^s_{u,p,q}(\R)}^{\spadesuit}_K := \norm{f \sep \mathcal{M}^u_p(\R)} + \abs{f}_K.
    $$
    If even $s>\sigma_{p}$, then the implied constants do not depend on $K$ and the assertion remains valid for $K=\infty$.
\end{theorem}

\begin{remark}
    We stress that our proof below only shows that $\norm{\cdot \sep \mathcal{E}^s_{u,p,q}(\R)}^{\spadesuit}_K$ defines an equivalent (quasi-)norm \emph{in} $\mathcal{E}^s_{u,p,q}(\R)$.  
    It does not guarantee that $f\in\mathcal{S}'(\R)$ belongs to $\mathcal{E}^s_{u,p,q}(\R)$ if $\norm{f \sep \mathcal{E}^s_{u,p,q}(\R)}^{\spadesuit}_K< \infty$. Indeed, for general $f\in\mathcal{S}'(\R)$ the latter expression is not even well-defined.
    Hence, in contrast to \autoref{thm:characterization} above, we do not claim to have a characterization. 
\end{remark}
\begin{proof}[Proof (of \autoref{thm:characterization_by_M})]
    Let $f\in\mathcal{E}^s_{u,p,q}(\R)$, where $s>\frac{p}{u} \, \sigma_{p}$.    
    Then standard embeddings (see, e.g., \cite[p.50]{HaMoSkr2020}) and \autoref{prop:E_in_M}(ii) yield that for sufficiently small $\delta>0$, we have
    \begin{align}\label{eq:proof_embeddings}
        \mathcal{E}^s_{u,p,q}(\R) 
        \hookrightarrow \mathcal{E}^{s-\delta}_{u,p,1}(\R)
        \hookrightarrow \mathcal{E}^{\frac{p}{u}\sigma_p}_{u,p,2}(\R)
        \hookrightarrow \mathcal{M}^u_p(\R).
    \end{align}
    Together with \autoref{lem:NormK}(i) this implies $\norm{f \sep \mathcal{E}^s_{u,p,q}(\R)}^{\spadesuit}_{-1} 
    \lesssim \norm{f \sep \mathcal{E}^s_{u,p,q}(\R)}$
as well as
\begin{align*}
    \norm{f \sep \mathcal{E}^s_{u,p,q}(\R)}^{\spadesuit}_K
    &\lesssim \norm{f \sep \mathcal{E}^s_{u,p,q}(\R)} + \bigg( \sum_{\ell=0}^K 2^{\ell(\sigma_p-s) \mu} \bigg)^{\frac{1}{\mu}} \norm{f \sep \mathcal{E}^s_{u,p,q}(\R)} \\
    &\lesssim \bigg( \sum_{\ell=0}^K 2^{\ell(\sigma_p-s) \mu} \bigg)^{\frac{1}{\mu}} \norm{f \sep \mathcal{E}^s_{u,p,q}(\R)}
\end{align*}
for every $K\in\N_0\cup\{\infty\}$. On the other hand, for $k\in\N$ and $u_k:=\mathcal{F}^{-1}(\varphi_k\,\mathcal{F}f) \in \mathcal{S}'(\R)$ the dyadic annuli criterion (\autoref{prop:dyadic_crit}) as well as \autoref{lem:M}(i) and \autoref{lem:NormK}(i) show that $u_k$ belongs to $\mathcal{E}^s_{u,p,q}(\R)$ with
\begin{align*}
    \norm{u_k \sep \mathcal{E}^s_{u,p,q}(\R)}
    &\lesssim \norm{2^{ks} u_k \sep \mathcal{M}^u_p(\R)} \\
    &\leq \norm{ \bigg( \sum_{k=1}^\infty 2^{ksq} \abs{u_k(\cdot)}^q \bigg)^{\frac{1}{q}} \sep \mathcal{M}^u_p(\R)} \\
    &= \abs{f}_{-1} \\
    &\lesssim \norm{f \sep \mathcal{E}^s_{u,p,q}(\R)}.
\end{align*}
Due to \eqref{eq:proof_embeddings} from this it also follows that $u_k\in \mathcal{M}^u_p(\R)$, where
\begin{align}\label{eq:proof_bound_uk}
    \norm{u_k \sep \mathcal{M}^u_p(\R)} 
    = 2^{-ks} \norm{2^{ks} u_k \sep \mathcal{M}^u_p(\R)}
    \leq 2^{-ks} \abs{f}_{-1}, \qquad k\in\N.
\end{align}
Similarly $f_0 := \mathcal{F}^{-1}(\varphi_0\,\mathcal{F}f) \in \mathcal{S}'(\R)$ belongs to $\mathcal{E}^s_{u,p,q}(\R) \hookrightarrow \mathcal{M}^u_p(\R)$ as \autoref{prop:dyadic_crit} and \eqref{eq:proof_K} imply $\norm{f_0 \sep \mathcal{E}^s_{u,p,q}(\R)}
    \lesssim \norm{f_0 \sep \mathcal{M}^u_p(\R)}
    \lesssim \norm{f \sep \mathcal{E}^s_{u,p,q}(\R)}$.
Next we recall that $f = f_0 + \sum_{k=1}^\infty u_k$ in $\mathcal{E}^{s-\delta}_{u,p,1}(\R)$ for every $\delta>0$; see again \autoref{prop:dyadic_crit}. 
Therefore, \eqref{eq:proof_embeddings} shows that in $\mathcal{M}^u_p(\R)$ there holds $f_0 = f - \sum_{k=1}^\infty u_k$ with
\begin{align*}
    \norm{f_0 \sep \mathcal{M}^u_p(\R)}
    &\lesssim \norm{f \sep \mathcal{M}^u_p(\R)} + \norm{\sum_{k=1}^\infty u_k \sep \mathcal{M}^u_p(\R)}
    \lesssim \norm{f \sep \mathcal{M}^u_p(\R)} + \abs{f}_{-1},
\end{align*}
where we applied \autoref{lem:M}(iv) with $\tau:=\min\{1,p\}$ as well as \eqref{eq:proof_bound_uk} to estimate
\begin{align*}
    \norm{\sum_{k=1}^\infty u_k \sep \mathcal{M}^u_p(\R)}^\tau
    \leq \sum_{k=1}^\infty \norm{ u_k \sep \mathcal{M}^u_p(\R)}^\tau
    \leq  \sum_{k=1}^\infty 2^{-ks\tau} \, \abs{f}_{-1}^\tau 
    \lesssim \abs{f}_{-1}^\tau.
\end{align*}
For $K\in\N_0\cup\{-1,\infty\}$ we can thus employ \eqref{eq:proof_K} and $\abs{f}_{-1}\leq \abs{f}_K$ to finally conclude that
\begin{align*}
    \norm{f \sep \mathcal{E}^s_{u,p,q}(\R)}
    &\lesssim \norm{f \sep \mathcal{M}^u_p(\R)} + \abs{f}_{-1}
    \leq \norm{f \sep \mathcal{M}^u_p(\R)} + \abs{f}_{K}
    = \norm{f \sep \mathcal{E}^s_{u,p,q}(\R)}_K^{\spadesuit}.
\end{align*}
The proof is complete.
\end{proof}

\vspace{0,3 cm}

\noindent\textbf{Acknowledgments. } The authors like to thank Cornelia Schneider, Petru Cioica-Licht and Joshua Kortum for some valuable discussions.

\phantomsection

\end{document}